\documentclass[12pt]{amsart}

\usepackage[margin=1in]{geometry}
\usepackage{amsmath,amssymb,amsthm,graphicx}
\usepackage[format=hang,justification=raggedright,margin=0pt]{subfig}

\newtheorem{theorem}{Theorem}[section]
\newtheorem{corollary}[theorem]{Corollary}
\newtheorem{lemma}[theorem]{Lemma}
\newtheorem{proposition}[theorem]{Proposition}
\theoremstyle{definition}
\newtheorem{definition}[theorem]{Definition}
\newtheorem{remark}[theorem]{Remark}

\newcommand{\abs}[1]{\left\lvert #1 \right\rvert}
\newcommand{\norm}[1]{\left\lVert #1 \right\rVert}
\newcommand{\angularbrack}[1]{\left\langle #1 \right\rangle}

\title{Duality Properties of Indicatrices of Knots}

\author[Adams]{Colin Adams}
\address{Colin Adams, Department of Mathematics and Statistics, Williams College, Williamstown, MA 01267}
\email{Colin.C.Adams@williams.edu}

\author[Collins]{Dan Collins}
\address{Dan Collins, Department of Mathematics, Princeton University, Princeton, NJ 08544}
\email{djc224@cornell.edu}

\author[Hawkins]{Katherine Hawkins}
\address{Katherine Hawkins, Department of Mathematics, Episcopal High School, P.O. Box 271299, Houston, TX 77277}
\email{khawkins@ehshouston.org}

\author[Sia]{Charmaine Sia}
\address{Charmaine Sia, Department of Mathematics, Harvard University, Cambridge, MA 02138}
\email{sia@mit.edu}

\author[Silversmith]{Rob Silversmith}
\address{Rob Silversmith, Department of Mathematics, 530 Church St., University of Michigan, Ann Arbor, MI 48109-1043}
\email{rsilvers@mich.edu}

\author[Tshishiku]{Bena Tshishiku}
\address{Bena Tshishiku, Department of Mathematics, University of Chicago, Chicago, IL 60637}
\email{tshishikub10@mail.wlu.edu}

\thanks{Research supported by NSF grant DMS-0850577 and funds provided by Williams College.  This paper benefited from previous work by the SMALL Knot Theory group at Williams College in summer, 2008, which consisted of Colin Adams, William George, Rachel Hudson, Ralph Morrison, Laura Starkston, Samuel Taylor, and Olga Turanova.  The authors would also like to thank Allison Henrich for helpful discussions.}



\date{\today}

\begin{document}

\begin{abstract}
	The bridge index and superbridge index of a knot are important invariants in knot theory.  We define the bridge map of a knot conformation, which is closely related to these two invariants, and interpret it in terms of the tangent indicatrix of the knot conformation.  Using the concepts of dual and derivative curves of spherical curves as introduced by Arnold, we show that the graph of the bridge map is the union of the binormal indicatrix, its antipodal curve, and some number of great circles.  Similarly, we define the inflection map of a knot conformation, interpret it in terms of the binormal indicatrix, and express its graph in terms of the tangent indicatrix.  This duality relationship is also studied for another dual pair of curves, the normal and Darboux indicatrices of a knot conformation.  The analogous concepts are defined and results are derived for stick knots.
\end{abstract}

\maketitle

\section{Introduction} \label{sec:introduction}

	The bridge index and superbridge index of a knot are important invariants in knot theory.  They are defined as the minimum and maximum number of local maxima respectively in the projection of a conformation of a knot onto an axis, minimized over all conformations of the knot.  The bridge index was introduced by Schubert~\cite{Sc54} in~1954 to study companionship in satellite knots.  The superbridge index was introduced by Kuiper~\cite{Ku87} in~1987, and has proven to be a useful invariant in obtaining lower bounds on the stick number of knots, which is defined as the least number of straight line segments needed to be placed end-to-end to form the knot in space.  For example, Jin~\cite{Ji97} used the superbridge index to prove that the stick number of the $(p, q)$-torus knot~$T_{p, q}$ is equal to $2q$ for $2 \leq p < q < 2p$.

	For a fixed embedding~$K$ of a knot into $\mathbb{R}^3$ and any choice of axis, we can count the number of stationary points in the projection of the knot onto that axis.  This allows us to define a map on the unit sphere, the bridge map of the knot conformation~$K$, by assigning to each vector $\mathbf{v} \in S^2$  one half of the number of stationary points in the projection of the knot onto the axis defined by~$\mathbf{v}$.  The bridge index is then just one half of the minimum value of the bridge map, minimized over all conformations of the knot. The superbridge index is one half the maximum value of the bridge map, minimized over all conformations of the knot.
	
	In this paper, we study the graph of the bridge map, the (minimal) set of points in~$S^2$ that separates the sphere into open regions on which the bridge map is constant.  By interpreting the bridge map in terms of intersections of the tangent indicatrix of $K$ with great circles and using the concepts of dual and derivative curves as introduced by Arnold~\cite{Ar95} in the study of the geometry of spherical curves, we show that the bridge graph is the union of the binormal indicatrix of the knot conformation with its antipodal curve and a number of great circles.  We also define another map on~$S^2$, the inflection map of the knot conformation~$K$, which counts the number of inflection points and cusps in planar projections of~$K$, and show that its graph is the union of the tangent indicatrix of the knot conformation with its antipodal curve and a number of great circles.  This duality relationship is studied for another dual pair of curves, the normal and Darboux indicatrices of a knot conformation.  Analogous concepts are defined and results are derived for stick knots.

	This paper is organized as follows.  In Sec.~\ref{sec:preliminaries_smooth}, we review the necessary background from the differential geometry of closed space curves, define the bridge and inflection maps and their respective graphs for smooth knot conformations, and relate these maps to intersections of the spherical indicatrices with great circles.  In Sec.~\ref{sec:smooth_knots}, we introduce the concepts of dual and derivative curves as introduced by Arnold~\cite{Ar95} and prove the smooth knot versions of our main results.  In Sec.~\ref{sec:preliminaries_polygonal}, we review McRae's~\cite{Mc01} definition of dual and derivative curves of spherical polygons and Banchoff's~\cite{Ba82} construction of the spherical indicatrices of space polygons.  In Sec.~\ref{sec:polygonal_knots}, we derive analogues of the results in Sec.~\ref{sec:smooth_knots} for stick knots.  Finally, in Sec.~\ref{sec:discussion}, we discuss possible methods of solving existing open problems using the ideas developed in this paper.

\section{Preliminaries} \label{sec:preliminaries_smooth}
	
	Throughout this paper, we shall use the notation~$K$ to denote a particular embedding of a knot in $\mathbb{R}^3$, and $[K]$ to denote a knot type, that is, the equivalence class of embeddings that can be obtained from a particular one under ambient isotopy.  However, as in common parlance, we shall often refer to a knot type as a knot, and an embedding of a knot as a knot conformation.  We shall also use the terms ``smooth knot'' and ``stick knot'' when referring to smooth and polygonal conformations of a knot respectively.

	We follow Fenchel's~\cite{Fe51} definition of the Frenet trihedral and curvature.  Let~$s$, $0 \leq s \leq l$, denote the arc length and $\mathbf{r}(s)$ the position vector of a varying point on a space curve $K$.  We assume that the coordinates of $\mathbf{r}(s)$ are smooth.  Each point $\mathbf{r}(s)$ has an associated osculating plane, that is, a plane containing the tangent vector $\mathbf{t} = \mathbf{r}'(s)$ and the vector $\mathbf{r}''(s)$. We assume that a suitably oriented unit vector normal to the plane, $\mathbf{b}(s)$, the binormal vector of $K$, has smooth coordinates.  We also assume that the vectors $\mathbf{t}'$ and~$\mathbf{b}'$ do not vanish simultaneously, and in addition, that they vanish at only a finite number of points.  This implies that no arc of $K$ is contained in a plane.  This allows us to avoid the usual assumption that $\mathbf{r}''(s)$, and hence the curvature, never vanishes.  Finally, we make the additional assumption that all of the vectors in the collection of vectors $\pm \mathbf{t}$ at the set of points where $\mathbf{t}'$ vanishes are pairwise distinct, and similarly for~$\pm \mathbf{b}$ at the set of points where $\mathbf{b}'$ vanishes .  The geometric meaning of this assumption will be explained in Sec.~\ref{sec:smooth_knots}.

	We proceed to define the unit normal vector by
	\[\mathbf{n} = \mathbf{b} \times \mathbf{t}\]
and the curvature $\kappa$ and torsion $\tau$ by
	\[\mathbf{t}' = \kappa \mathbf{n} \qquad \textrm{and} \qquad \mathbf{b}' = -\tau \mathbf{n},\]
which is possible since the derivatives are orthogonal to both $\mathbf{t}$ and~$\mathbf{b}$.  The definition of the normal vector~$\mathbf{n}$ then yields
	\[\mathbf{n}' = -\kappa\mathbf{t} + \tau\mathbf{b}.\]
Thus we obtain the usual Frenet-Serret formulas for the movement of a point along a smooth curve:
	\begin{align*}
	\mathbf{t}' & = \kappa \mathbf{n}, \\
	\mathbf{n}' & = -\kappa\mathbf{t} + \tau\mathbf{b}, \\
	\mathbf{b}' & = - \tau \mathbf{n}.
	\end{align*}
Note, however, that in this circumstance,  the curvature may vanish and even be negative.  However, only a change in the sign of $\kappa$, and not the sign itself, has a geometric significance, since we may replace $\mathbf{b}$ and hence~$\mathbf{n}$ by its opposite vector (for all~$s$).  Points where $\kappa$ and~$\tau$ change sign will be called $\kappa$- and $\tau$-inflection points respectively.

	The vectors $(\mathbf{t}, \mathbf{n}, \mathbf{b})$ form a trihedral known as the \emph{Frenet trihedral} or \emph{Frenet frame}.  If this trihedral is parallel translated to the origin, the endpoints of the translated vectors describe three curves $T$, $N$, and~$B$ on the unit sphere $S^2 \in \mathbb{R}^3$, which are called the \emph{tangent}, \emph{normal}, and \emph{binormal indicatrices} of the curve respectively (or tantrix, notrix, and binotrix respectively for short).  The lengths of these indicatrices are given by the formulae
	\begin{align*}
	s_T & = \int_K \abs{\kappa(s)} \; \mathrm{d}s \\
	s_N & = \int_K \sqrt{\kappa^2(s) + \tau^2(s)} \; \mathrm{d}s \\
	s_B & = \int_K \abs{\tau(s)} \; \mathrm{d}s. \\
	\end{align*}
Note that the above formulae differ from the regular formulae by the addition of the absolute value symbol around $\kappa(s)$, which arises because we have allowed~$\kappa$ to take on negative values.

	The Frenet trihedral defines a rigid motion around the origin called \emph{Frenet motion}, with angular velocity $\omega = \omega(s)$.  The instantaneous axis of rotation of Frenet motion, which we call the \emph{Darboux axis}, lies in the normal plane containing $\mathbf{b}$ and~$\mathbf{t}$ because it is perpendicular to the velocity vectors $\mathbf{b}'$ and~$\mathbf{t}'$.  The unit Darboux vector~$\mathbf{d}$ is defined to be the unit vector on the Darboux axis such that its sense, together with the sense of rotation of the Frenet frame, form a right-handed screw.  From the orthogonality of $\mathbf{d}$ and~$\mathbf{n}'$, we obtain
	\[\omega\mathbf{d} = \mathbf{n} \times \mathbf{n'} = \kappa \mathbf{b} + \tau \mathbf{t}.\]
As $s$ varies, the endpoint of the Darboux vector describes yet another curve on~$S^2$, the Darboux indicatrix~$D$ of the curve.

	Next, we define the bridge map and bridge graph of a knot conformation and relate it to two knot invariants, the bridge index and superbridge index of a knot.  We also define the inflection map of a knot conformation and its corresponding graph.

\begin{definition}
	Given a smooth knot conformation~$K$ in $\mathbb{R}^3$, the \emph{bridge map} of the knot conformation~$K$ is the map defined on the unit sphere~$S^2$ by
	\[\mathfrak{b}_\mathbf{v}(K) = \#\{\mbox{stationary points in the projection of $K$ onto the axis defined by $\mathbf{v}$}\}\]
for every vector $\mathbf{v} \in S^2$.
	\label{def:bridgemap}
\end{definition}

	The bridge map is closely related to two knot invariants, the bridge index and superbridge index of a knot.

\begin{definition}[Schubert~\cite{Sc54}]
	The \emph{bridge index} of a knot type~$[K]$ is defined by
	\[b[K] = \min_{K \in [K]} \min_{\mathbf{v} \in S^2} \#\{\mbox{local maxima in the projection of $K$ onto the axis defined by $\mathbf{v}$}\}.\]
\end{definition}

	It is easy to see that the bridge index of a knot is related to the bridge map by
	\[b[K] = \frac{1}{2} \min_{K \in [K]} \min_{\mathbf{v} \in S^2} \mathfrak{b}_\mathbf{v}(K).\]

	The bridge index has been extensively studied.  Schubert~\cite{Sc54} proved that bridge index is additive minus one under composition: $b[K_1 \# K_2] = b[K_1] + b[K_2] - 1$. Knots with bridge index 2 are precisely the rational knots.
	
	Kuiper~\cite{Ku87} introduced a related knot invariant, the superbridge index of a knot.

\begin{definition}[Kuiper~\cite{Ku87}]
	The \emph{superbridge index} of a knot type~$[K]$ is defined by
	\[sb[K] = \min_{K \in [K]} \max_{\mathbf{v} \in S^2} \#\{\mbox{local maxima in the projection of $K$ onto the axis defined by $\mathbf{v}$}\}.\]
\end{definition}

	Similarly, the superbridge index of a knot type is related to the bridge map by
	\[b[K] = \frac{1}{2} \min_{K \in [K]} \max_{\mathbf{v} \in S^2} \mathfrak{b}_\mathbf{v}(K).\]

	The superbridge index is a useful invariant in obtaining lower bounds on the stick number of knots.  Jin~\cite{Ji97} used the superbridge index to show that the stick number of the $(p, q)$-torus knot~$T_{p, q}$ is equal to~$2q$ for $2 \leq p < q < 2p$.

\begin{definition}
	Given a smooth knot conformation~$K$ in~$\mathbb{R}^3$, the \emph{inflection map} of the knot conformation~$K$ is the map defined on the unit sphere~$S^2$ by
	\[\mathfrak{i}_\mathbf{v}(K) = \#\{\mbox{inflection points and cusps in the projection of $K$ onto the plane orthogonal to $\mathbf{v}$}\}\]
for every vector $\mathbf{v} \in S^2$.
	\label{def:inflectionmap}
\end{definition}

\begin{remark}
	Observe that a point~$p$ on the knot conformation~$K$ projects to a stationary point along the axis defined by~$\mathbf{v}$ if and only if the tangent vector to~$K$ at~$p$ is orthogonal to~$\mathbf{v}$.  Hence the stationary points in the projection of~$K$ onto the axis defined by~$\mathbf{v}$ are in bijective correspondence with the points of intersection of the tantrix of~$K$ with the great circle orthogonal to~$\mathbf{v}$.  Similarly, since~$p$ projects to an inflection point or a cusp in the plane orthogonal to~$\mathbf{v}$ if and only if the osculating plane at~$p$ is projected onto a line, which occurs if and only if the binormal vector at~$p$ is orthogonal to~$\mathbf{v}$, it follows that the inflection points and cusps in the projection of~$K$ onto the plane orthogonal to~$\mathbf{v}$  are in bijective correspondence with the points of intersection of the binotrix of~$K$ with the great circle orthogonal to~$\mathbf{v}$.  We shall make extensive use of this observation in the proof of Theorem~\ref{thm:bridgeinflection}.
	\label{rem:intersectionswithgreatcircles}
\end{remark}

\begin{lemma}[see, e.g., Blaschke~\cite{Bl37}]
	The length of a curve~$C$ on the unit sphere is equal to $\pi$~times the average over all great circles~$G$ of the number of times that $C$ intersects~$G$.
	\label{lem:intersectionswithgreatcircles}
\end{lemma}

     We interpret the word ``average'' in the above lemma as follows.  Corresponding to each great circle~$G$, there is a unique pair of unit vectors that are perpendicular to the plane containing $G$.  The usual Lebesgue measure on the unit sphere applied to them then induces a measure on the set of all great circles.  (The average is then interpreted as a Lebesgue integral.)

	Lemma~\ref{lem:intersectionswithgreatcircles}, together with Remark~\ref{rem:intersectionswithgreatcircles} and the fact that the bridge map is not constant for any conformation~$K$ of a non-trivial knot, implies that
	\[\textrm{total (absolute) curvature of $K$} = \textrm{length of tantrix of $K$} = \pi\angularbrack{\mathfrak{b}_\mathbf{v}(K)} > 2\pi b[K],\]
where $\angularbrack{\mathfrak{b}_\mathbf{v}(K)}$ denotes the average of the bridge map~$\mathfrak{b}_\mathbf{v}(K)$ over all $\mathbf{v} \in S^2$ under the usual Lebesgue measure.  Since the bridge index of any non-trivial knot is at least~$2$, this, as noted by Milnor in~\cite{Mi53}, yields the well-known theorem of F\'{a}ry~\cite{Fa49} and Milnor~\cite{Mi50} that the total (absolute) curvature of a conformation of any non-trivial knot is greater than~$4\pi$.

	The following definition will be used extensively in the statement of our main results.

\begin{definition}
	The \emph{graph} of a map on the unit sphere~$S^2$ is the set of points $p \in S^2$ such that there does not exist an open neighborhood~$N_p$ of~$p$ on~$S^2$ (with the standard Euclidean topology) such that the value of the map is constant for all points $q \in N_p$ at which the map is defined.
	\label{def:graph}
\end{definition}

	This definition can be interpreted as follows.  Label each point $\mathbf{v} \in S^2$ with the value of the map at that point.  This divides the sphere into several regions with the same label, and the interiors of those regions are separated by the graph of the map.

	The graphs of the bridge map and the inflection map, which we call the \emph{bridge graph} and \emph{inflection graph}, as well as two other graphs that we shall define later on, the tantrix-bridge graph and tantrix-inflection graph, form the subject of the remainder of this paper.

\section{Duality Relationships for Indicatrices of Smooth Knots} \label{sec:smooth_knots}

	In this section, we review the concepts of the dual curve and derivative curve of a co-oriented curve on a sphere as introduced by Arnold~\cite{Ar95} and use them to determine the bridge graph and inflection graph of smooth knots.  We also use the concepts of dual and derivative curves to study the relationships between another dual pair of curves, the notrix and Darboux indicatrix.

	For our purposes, a co-orientation of a vector in a plane is a choice of one of the two unit vectors perpendicular to it in the plane, and a co-orientation of a spherical curve is a continuous choice of co-orientations of tangent vectors in their tangent planes.  A wave front is a curve obtained from a smooth co-oriented curve by moving each point of the curve by a constant distance along the co-orienting normal.  We refer the reader to~\cite{Ar95} for a more general definition of co-orientations and wave fronts in the context of contact geometry.

\begin{definition}[Arnold~\cite{Ar95}]
	The \emph{dual curve}~$\Gamma^\vee$ to a given co-oriented curve~$\Gamma$ on the sphere is the curve obtained from the original curve by moving a distance of~$\pi/2$ along the normals on the side determined by the co-orientation.  The dual curve~$\Gamma^\vee$ inherits its co-orientation from that of~$\Gamma$.
	\label{def:dual}
\end{definition}

	This definition applies not only to smoothly immersed curves, but also to wave fronts having semi-cubical cusps.  The cusps on the original curve correspond to points of inflection on the dual curve, while points of inflection on the original curve correspond to cusps on the dual curve.

\begin{definition}[Arnold~\cite{Ar95}]
	The \emph{derivative curve}~$\Gamma'$ of a co-oriented curve~$\Gamma$ on the oriented standard sphere~$S^2$ is the curve obtained by moving each point a distance $\pi/2$ along the great circle tangent to the original curve at that point.  The direction of motion along the tangent is chosen so that the orientation of the sphere, given by the direction of the tangent and the direction of the co-orienting normal, is positive.
	\label{def:derivative}
\end{definition}

	For example, if the tantrix is co-oriented such that the co-orienting normal is obtained from the derivative of the tangent vector via a counterclockwise rotation of~$\pi/2$ on the surface of the sphere when the curvature~$\kappa$ is positive and via a clockwise rotation of~$\pi/2$ when $\kappa$~is negative, then it is easy to see that the derivative curve of the tantrix is the notrix and the dual curve of the tantrix is the binotrix.  Note that the co-orientation of the binotrix is thus induced by the co-orientation on the tantrix. In what follows, we shall always assume that the tantrix is oriented as such.  We refer the reader to Arnold's~\cite{Ar95} paper for a more thorough discussion of the geometry of spherical curves.

	A cusp of the tantrix and a spherical inflection point of the binotrix corresponds to a $\kappa$-inflection point of the knot conformation~$K$, while a spherical inflection point of the tantrix and a cusp of the binotrix corresponds to a $\tau$-inflection point of~$K$.  (See, e.g., Fenchel~\cite{Fe51} for more details.)  Our assumption in Sec.~\ref{sec:preliminaries_smooth} that the vectors in the collection of vectors $\pm \mathbf{t}$ at the set of points where $\mathbf{t}'$ vanishes are pairwise distinct and the vectors in the collection of vectors $\pm \mathbf{b}$ at the set of points where $\mathbf{b}'$ vanishes are pairwise distinct simply says that no two cusps of the tantrix coincide and a cusp of the tantrix never coincides with a cusp of the anti-tantrix, and similarly for the binotrix.

	The dual and derivative curves of a co-oriented spherical curve satisfy the following properties.

\begin{lemma}[Arnold~\cite{Ar95}]
The dual curve is formed from the centers of the great circles tangent to the original curve.  The second dual of a curve is antipodal to the original curve: $\Gamma^{\vee\vee} = -\Gamma$.
	\label{lem:seconddual}
\end{lemma}

\begin{lemma}[Arnold~\cite{Ar95}]
The derivative of a curve coincides with the derivative of any curve equidistant from it and is a smoothly immersed curve on~$S^2$ even if the original curve has generic singularities.
	\label{lem:smoothderivative}
\end{lemma}

	Since the notrix and the Darboux indicatrix have the same derivative curves and the normal and Darboux vectors are mutually orthogonal, it follows from Lemma~\ref{lem:smoothderivative} that the notrix and Darboux indicatrix are dual curves.  In what follows, we shall always orient the notrix such that the co-orienting normal is obtained from the derivative of the normal vector via a counterclockwise rotation of~$\pi/2$ on the surface of the sphere; then the Darboux indicatrix is the dual curve of the notrix.  Since the notrix is the derivative curve of the tantrix, Lemma~\ref{lem:smoothderivative} tells us that it has no cusps  (this  also follows from the fact that $\omega = \sqrt{\kappa^2 + \tau^2)} > 0$ because $\kappa$ and~$\tau$ do not vanish simultaneously), and hence the Darboux indicatrix has no spherical inflection points.  Inflection points of the notrix and cusps of the Darboux indicatrix correspond to points where the geodesic curvature $\tau/\kappa$ of the tantrix is stationary and the knot behaves locally like a helix.  (See Fenchel~\cite{Fe51} or Uribe-Vargas~\cite{Ur04} for further details.)  In what follows, we shall also assume that no two cusps of the Darboux indicatrix coincide and a cusp of the Darboux indicatrix never coincides with a cusp of the anti-Darboux indicatrix.

	Since the notrix is the curve of normalized derivatives of the tantrix and the binotrix, we may view it as the ``tantrix of the tantrix'' or the ``tantrix of the binotrix'' up to a sign.  Similarly, since the Darboux indicatrix is the dual curve to the notrix, we may view it as the ``binotrix of the tantrix'' or the ``binotrix of the binotrix.''  In what follows, we shall focus our interpretation of the notrix and Darboux indicatrix as the tantrix and binotrix respectively of the tantrix; similar results hold if we interpret them as the tantrix and binotrix respectively of the binotrix.  This allows us to define a ``bridge map'' and ``inflection map'' for the tantrix analogous to Definitions~\ref{def:bridgemap} and~\ref{def:inflectionmap}.
	
\begin{definition} 	Given a smooth knot conformation~$K$ in $\mathbb{R}^3$, the \emph{tantrix-bridge map} of the knot conformation~$K$ is the map defined on the unit sphere~$S^2$ by
	\[\mathfrak{b}_\mathbf{v}(K) = \#\{\mbox{stationary points in the projection of $T$ onto the axis defined by $\mathbf{v}$}\}\]
for every vector $\mathbf{v} \in S^2$.
	\label{def:tanbridgemap}
\end{definition}	

\begin{definition} Given a smooth knot conformation~$K$ in~$\mathbb{R}^3$, the \emph{tantrix-inflection map} of the knot conformation~$K$ is the map defined on the unit sphere~$S^2$ by
	\[\mathfrak{i}_\mathbf{v}(K) = \#\{\mbox{inflection points and cusps in the projection of $T$ onto the plane orthogonal to $\mathbf{v}$}\}\]
for every vector $\mathbf{v} \in S^2$.
	\label{def:taninflectionmap}
\end{definition}	
	
In analogy to Remark  \ref{rem:intersectionswithgreatcircles},  we can  reinterpret these maps in terms of intersections of the corresponding great circles with the notrix and Darboux indicatrix respectively.  Moreover, we can define the \emph{tantrix-bridge graph} and \emph{tantrix-inflection graph} as in Definition  \ref{def:graph}.  However, some care must be taken at cusps of the tantrix: when do we count a cusp of the tantrix as contributing to a stationary point in height to the tantrix-bridge map?  Returning to our interpretation of the tantrix-bridge map as the number of intersections of the notrix with a great circle, we see we want the derivative of the tantrix to lie in the plane of that great circle, and hence a cusp of the tantrix should be counted as a stationary point in height only if the derivative at that point is perpendicular to the height axis.  Similar arguments show that the only cusps that should contribute to an inflection point or cusp in the tantrix-inflection map are those that lie on the great circle in the projection.

	In the following two theorems, we determine the bridge graph, the inflection graph, the tantrix-bridge graph and the tantrix-inflection graph for smooth knots.

\begin{theorem}
	The bridge graph of a smooth knot is the union of the binotrix, the anti-binotrix, and the great circles tangent to the binotrix and anti-binotrix at points corresponding to $\kappa$-inflection points of the knot.  The inflection graph of a smooth knot is the union of the tantrix, the anti-tantrix, and the great circles tangent to the tantrix and anti-tantrix at points corresponding to $\tau$-inflection points of the knot.
	\label{thm:bridgeinflection}
\end{theorem}

\begin{proof}
	We prove the first statement.  There are three types of intersections of the tantrix with a great circle, in the vicinity of which the number of intersections of the tantrix with a great circle may change: (i) the great circle is tangential to the tantrix neither at a spherical inflection point of the tantrix nor at a cusp (Fig~\ref{fig:tantrix_tangential}), (ii) the great circle is tangential to the tantrix at a spherical inflection point of the tantrix (Fig~\ref{fig:tantrix_sphericalinflection}), and (iii) the great circle intersects the tantrix at a semi-cubical cusp of the tantrix (Fig~\ref{fig:tantrix_cusp}).	  
	\begin{figure}[ht]
		\centering
		\subfloat[Great circle is tangential to tantrix neither at a spherical inflection point of the tantrix nor at a cusp.]{\label{fig:tantrix_tangential}\includegraphics[width=0.3\textwidth]{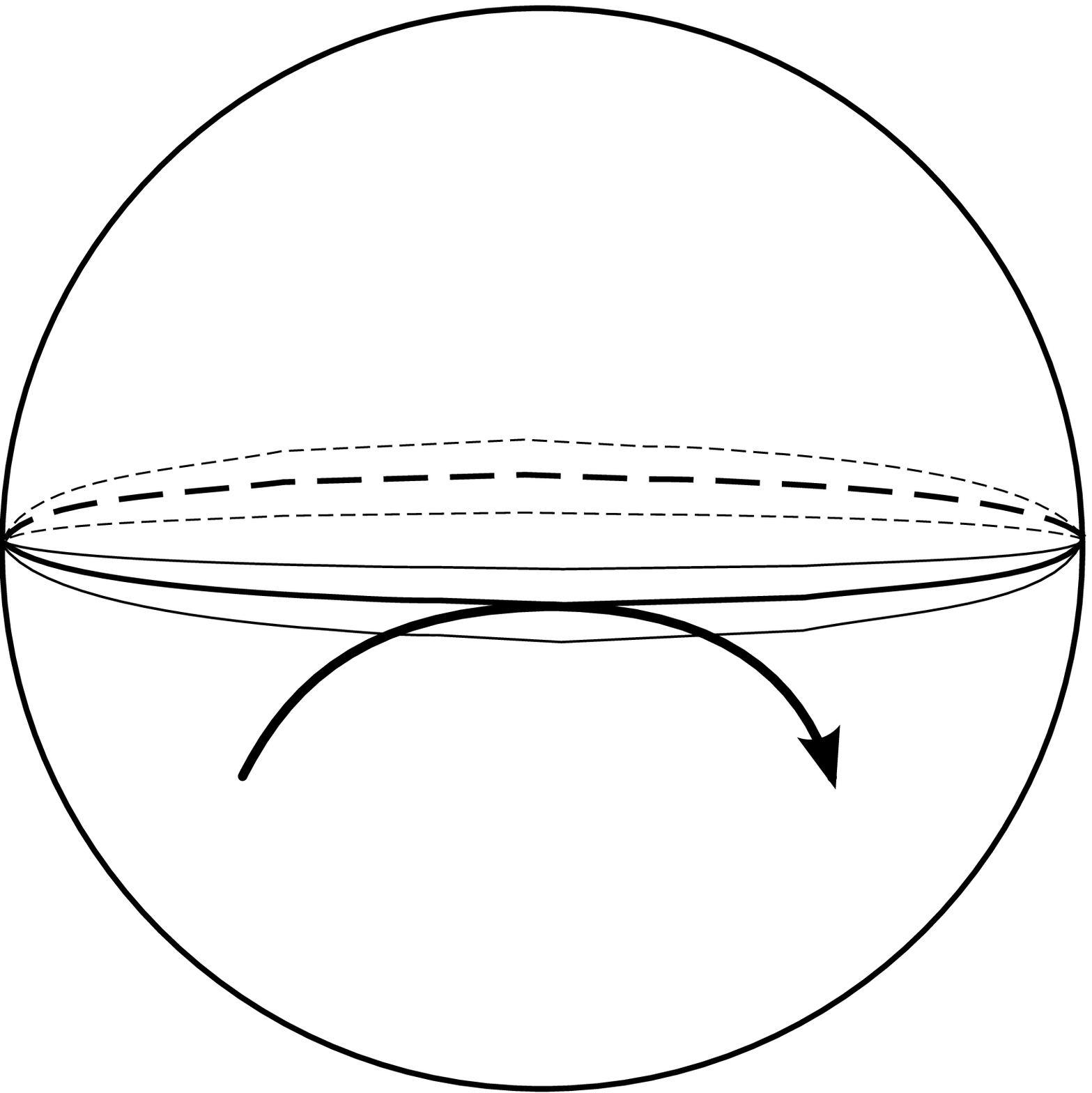}}\quad                
		\subfloat[Great circle is tangential to the tantrix at a spherical inflection point of the tantrix.]{\label{fig:tantrix_sphericalinflection}\includegraphics[width=0.3\textwidth]{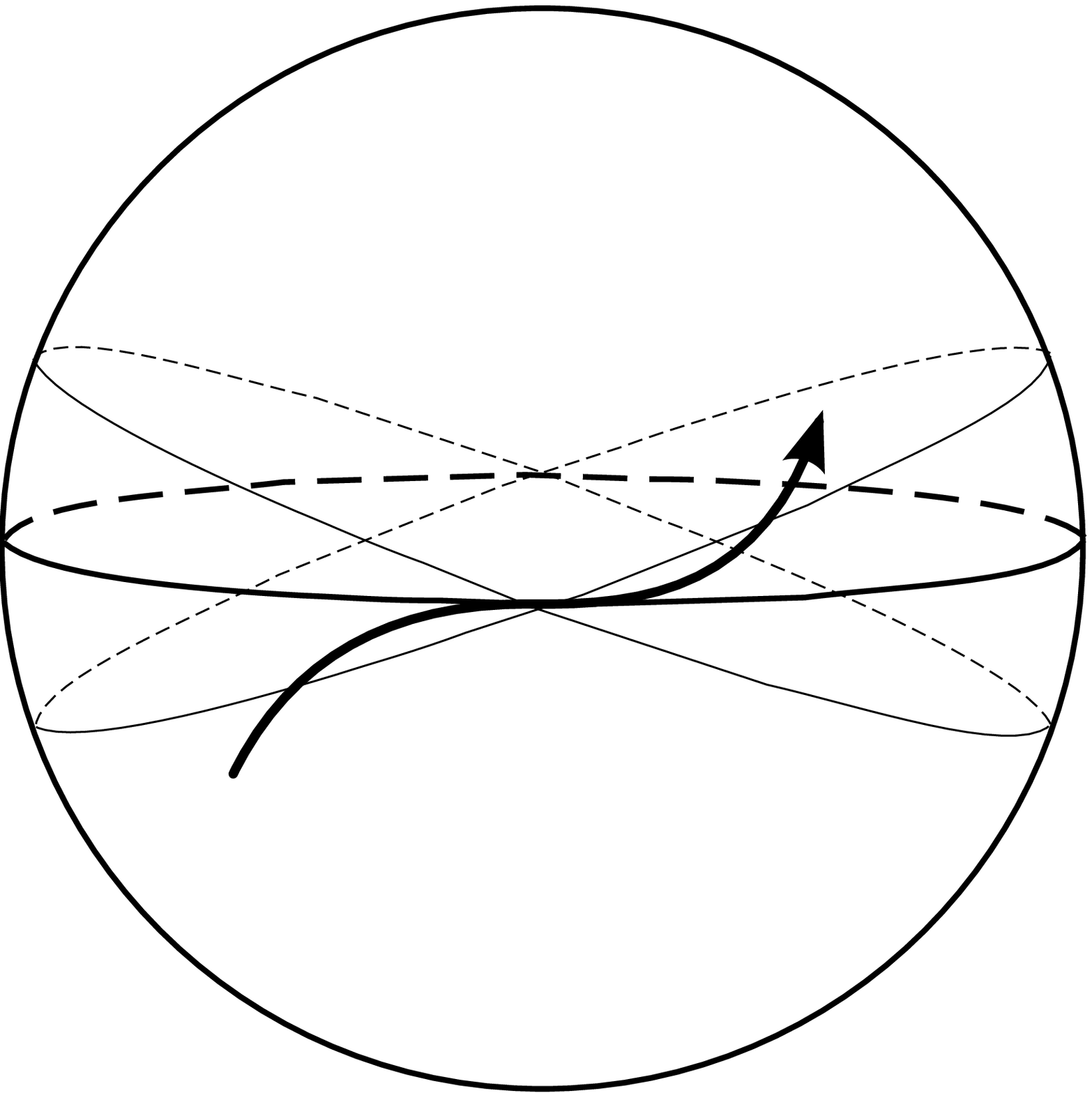}}\quad
		\subfloat[Great circle intersects the tantrix at a semi-cubical cusp of the tantrix.]{\label{fig:tantrix_cusp}\includegraphics[width=0.3\textwidth]{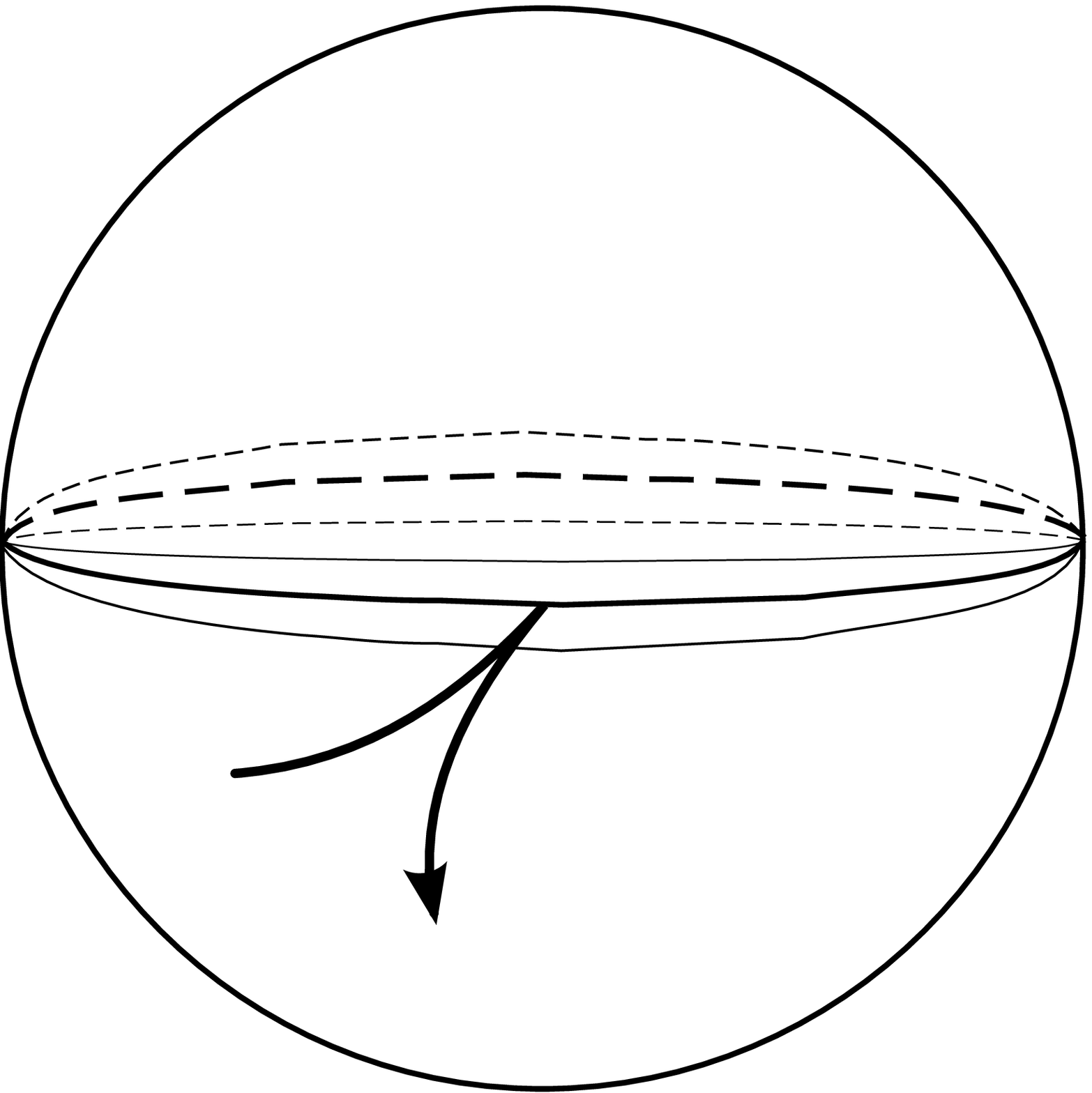}}
		\caption{Types of intersections of the tantrix with a great circle, in the vicinity of which the number of intersections of the tantrix with a great circle may change.}
		\label{fig:tantrix_intersections}
	\end{figure}

	Intersections of the first type contribute a portion of the bridge graph that is formed from the centers of the great circles tangent to the tantrix at these points, which by Lemma~\ref{lem:seconddual} is the dual curve to the tantrix and its antipodal curve, the binotrix and anti-binotrix, excluding cusp points on both curves.  Next, a change in the number of intersections of the tantrix with a great circle in the vicinity of a spherical inflection point of the tantrix corresponds to moving the center of a great circle from outside a cusp of the binotrix to inside (and similarly for the anti-binotrix), or vice versa, as shown in Fig.~\ref{fig:binotrix_cusp}, as the limiting tangent at a cusp of the binotrix is obtained from the tangent to the tantrix at a spherical inflection point by moving a distance of~$\pi/2$ along the normal on the side determined by the co-orientation.  Finally, the number of intersections of the tantrix with a great circle changes as we move the great circle over a cusp in almost every direction (by our assumption that no two cusps of the tantrix coincide and a cusp of the tantrix never coincides with a cusp of the anti-tantrix).  Hence the portion of the bridge graph that such intersections contribute to are great circles whose centers are at those cusps, and which are therefore great circles tangent to the binotrix and anti-binotrix at points corresponding to $\kappa$-inflection points of the knot.
	\begin{figure}[ht]
		\centering \includegraphics[width=0.3\textwidth]{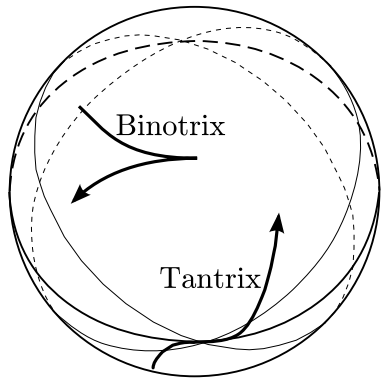}
		\caption{A change in the number of intersections of the tantrix with a great circle in the vicinity of a spherical inflection point of the tantrix corresponds to moving the center of a great circle from outside a cusp of the binotrix to inside or vice versa.}
		\label{fig:binotrix_cusp}
	\end{figure}
	
	The proof of the second statement follows analogously upon noting that cusps in the binotrix correspond to $\tau$-inflection points of the knot.
\end{proof}

\begin{remark}
	In fact, a little additional work enables us to identify the direction in which the value of the bridge map changes whenever we cross the bridge graph.  In general, the value of the bridge map increases when we move across the bridge graph from a region that locally ``bulges out'' to an adjacent region that locally ``bulges in.''  The change in value of the bridge map as we move across a great circle defining the bridge graph can be determined from the shape of the binotrix near a $\kappa$-inflection point.  Similar observations apply to the inflection graph.  The details are relatively straightforward and will be omitted.
\end{remark}

	Since the notrix has no cusps while the Darboux indicatrix has cusps precisely when the geodesic curvature~$\tau/\kappa$ of the tantrix is stationary, a similar argument yields the following theorem.

\begin{theorem}
	The tantrix-bridge graph is the union of the Darboux indicatrix and the anti-Darboux indicatrix.  The tantrix-inflection graph is the union of the notrix, the anti-notrix, and the great circles tangent to the notrix and the anti-notrix at the points corresponding to where the geodesic curvature~$\tau/\kappa$ of the tantrix is stationary.
	\label{thm:tantrix-bridgeinflection}
\end{theorem}

\section{Duality for Spherical Polygons and Indicatrices of Stick Knots} \label{sec:preliminaries_polygonal}

	In this section, we review McRae's~\cite{Mc01} definition of the dual curve of a spherical polygon~$P$.  We also review Banchoff's~\cite{Ba82} construction of the tantrix, notrix, and binotrix of an oriented space polygon~$X$ in~$\mathbb{R}^3$ and propose a definition for the Darboux indicatrix of the oriented space polygon~$X$.  This section provides the necessary background for Sec.~\ref{sec:polygonal_knots}.

	The fundamental idea behind duality for spherical polygons is that the dual of a point on the unit sphere~$S^2$ is the co-oriented great circle at a distance of~$\pi/2$ away from that point, with the co-orienting normal pointing away from the point, while the dual of a co-oriented great circle is the point~$\pi/2$ away in the direction of the co-orienting normal.  Thus, composing the dual with itself gives the antipodal map on~$S^2$.

	Let $P$ be a spherical polygon determined by a cycle of $n$~vertices $(v_0, v_1, \ldots, v_{n-1})$, and co-oriented geodesic segments $(l_0, l_1, \ldots, l_{n-1})$, where $l_j$ is the shorter of the great circle arcs joining $v_j$ to $v_{j+1}$ (we consider vertex indices modulo the number of vertices, so that, in particular, $v_n = v_0$).
	
\begin{definition}[McRae~\cite{Mc01}]
	The \emph{dual curve~$P^\vee$ of the spherical polygon~$P$} is defined to be the co-oriented polygon determined by the cycle of vertices $(V_0, V_1, \ldots, V_{n-1})$ and sides $(L_0, L_1, \ldots, L_{n-1})$, where $V_j = l_j^\vee$ and $L_j$ is the geodesic segment joining $V_{j-1}$ to $V_j$ whose length is equal to the exterior angle at $v_j$ and with co-orientation induced from $v_j^\vee$.
\end{definition}

	Observe that the relation $P^{\vee\vee} = -P$ continues to hold in this setting.

	For a co-oriented spherical polygon~$P$, we can define a second spherical polygon in terms of~$P$ and its dual~$P^\vee$, as follows.

\begin{definition}[McRae~\cite{Mc01}]
	Let~$P$ be a co-oriented spherical polygon.  We define the \emph{direct sum} of~$P$ and its dual, denoted $P \oplus P^\vee$, to be the co-oriented spherical polygon constructed in the following manner: we regard $v_0$ as the north pole, so that the co-oriented segment $L_0$ lies on its equator.  This segment is rotated counterclockwise around~$v_0$ through an angle of~$\pi/2$.  Next, we regard the endpoint $V_0$ of~$L_0$ as the south pole and rotate the segment $l_0$ counterclockwise through an angle of~$\pi/2$.  We repeat this process, alternating between angles of the form $v_i$ and~$V_i$, until we arrive back at the vertex~$v_0$.
\end{definition}

	McRae proves the following proposition about the direct sum $P \oplus P^\vee$.

\begin{proposition}[McRae~\cite{Mc01}]
	Consecutive sides of $P \oplus P^\vee$ meet at alternating right angles.
\end{proposition}

	More will be said about the direct sum $P \oplus P^\vee$ later in this section.

	Next, we review Banchoff's~\cite{Ba82} construction of the tantrix, notrix, and binotrix of an oriented space polygon~$X$ in~$\mathbb{R}^3$ and define the Darboux indicatrix of~$X$.  We consider~$X$ to be determined by a cycle of vertices $(X_0, X_1, \ldots, X_{n-1})$, where vertex indices are taken modulo~$n$, that is, $X$ is an oriented closed curve~$X(t)$ defined on some closed interval $[a, b]$, $a = t_0 < t_1 < \cdots < t_{n-1} < t_n = b$, with $X_i = X(t_i)$ and $X$ linear in each subintervcal $[t_i, t_{i+1}]$.  We say that $X$ is in \emph{general position} if no four consecutive vertices of~$X$ lie in a plane.  In what follows, we consider only space polygons in general position.  Moreover, we shall assume that no two (undirected) edges of~$X$ and no two osculating planes of~$X$ are parallel.

	Milnor~\cite{Mi50} defined the curvature of the space polygon~$X$ at the vertex~$X_i$ to be the angle~$\theta_i$, $0 < \theta_i < \pi$, between the vectors $X_i - X_{i-1}$ and $X_{i+1} - X_i$, and the total curvature of the polygon~$X$ to be the angle sum $\sum_{i=0}^{n-1} \theta_i$.  In particular, note that the curvature of a space polygon is positive at every vertex.  We further define the torsion~$\tau_i$ of the space polygon~$X$ at the edge $X_{i-1}X_i$ to be the directed angle~$\phi_i$, $-\pi < \phi_i < \pi$, between the projections of the directed edges $X_{i-2}X_{i-1}$ and $X_iX_{i+1}$ when we project down the directed edge $X_{i-1}X_i$, and the total absolute torsion of~$X$ to be the sum $\sum_{i=1}^{n-1} \abs{\phi_i}$.

	The tantrix, notrix, binotrix, and Darboux indicatrix of $X$ are defined to be co-oriented polygons on the unit sphere with vertices as follows.  The vertices $T_i$ and $B_i$ of the tantrix and binotrix respectively are defined by
	\[T_i = \frac{X_i - X_{i-1}}{\norm{X_i - X_{i-1}}}, \quad i = 1, \ldots, n \qquad \textrm{and} \qquad T_0 = T_n,\]
	\[B_i = \frac{T_i \times T_{i+1}}{\norm{T_i \times T_{i+1}}}, \quad i = 0, \ldots, n-1 \qquad \textrm{and} \qquad B_0 = B_n,\]
so that both of these indicatrices have the same number of vertices as $X$.  (It will often expedite calculations to use an unnormalized version of the vertices of the binotrix, which we denote by $\widetilde{B}_i = T_i \times T_{i+1}$.)  This definition can be interpreted as follows.  Each directed edge of~$X$ gives rise to a vertex of the tantrix, and each vertex of~$X$, together with the preceding and following directed edges, gives rise to a vertex of the binotrix.  At the vertex~$X_i$, the tangent line sweeps counterclockwise through an angle of~$\theta_i$ in the oriented osculating plane, which corresponds to a great circle arc of length~$\theta_i$ connecting two adjacent vertices of the tantrix, and along the edge $X_{i-1}X_i$, the binormal vector sweeps through an angle of~$\abs{\phi_i}$, which corresponds to a great circle arc of length~$\abs{\phi_i}$ connecting two adjacent vertices of the binotrix.  Hence these definitions preserve the property that the total curvature of the space polygon~$X$ is equal to the length of its tantrix and its total absolute torsion is equal to the length of the binotrix.  Note also that the relation $B = T^\vee$ continues to hold under these definitions (recall that we defined the tantrix to be co-oriented such that the co-orienting normal is oriented towards the left when viewed from the exterior of the sphere).

	Observe that the segments of the direct sum $P \oplus P^\vee$ obtained by rotating segments~$l_i$ of~$P$ behave like derivative curves of the arcs of~$P$ in the sense of Definition~\ref{def:derivative}, while those obtained by rotating segments~$L_i$ of~$P^\vee$ behave like tantrix arcs corresponding to the vertices of a stick knot.  Hence it makes sense to think of $P \oplus P^\vee$ as the derivative curve of the spherical polygon~$P$, and we write $P' = P \oplus P^\vee$.

	We say a few words about how the torsion of a space polygon can be interpreted in terms of the tantrix and the binotrix.  Since the torsion~$\tau_i$ is defined along the edge $X_{i-1}X_i$ of the space polygon~$X$, we can also think of it as a property of the vertex $T_i$ of the tantrix, or of the edge $B_{i-1}B_i$ of the binotrix.  If we orient the unit sphere such that the directed tantrix arc $T_{i-1}T_i$ lies on the equator and turns counterclockwise when viewed from the north pole, it is easy to see that $\tau_i$~is positive if the tantrix arc $T_iT_{i+1}$ lies above the equator, and negative if $T_iT_{i+1}$ lies below the equator.  Moreover, by regarding the torsion~$\tau_i$ as a property of the edge $B_{i-1}B_i$ of the binotrix, we can speak about a change of sign in torsion at a vertex of the binotrix.  Further relations between the sign of the torsion, the tantrix, and the binotrix will be elucidated later in this section.

	The notrix~$N$ of~$X$ should take into account that associated to each edge of~$X$ there is one tangent direction but two binormal directions, and associated to each vertex of~$X$ there is one binormal direction but two tangent directions.  Consequently, we define the vertices~$N_i$ of~$N$ by
	\[N_{2i} = T_i \times B_i, \quad N_{2i+1} = T_{i+1} \times B_i, \quad i = 0, \ldots, n-1, \qquad \textrm{and} \qquad N_0 = N_{2n}.\]
	Note that the notrix has twice as many vertices as~$X$.  It is easy to see from the definition of the notrix that $N = T \oplus B = T \oplus T^\vee$, thus yielding an analogue of the result $N = T'$ for space polygons.  As in Sec.~\ref{sec:smooth_knots}, we co-orient the notrix such that the co-orienting normal is oriented towards the left when viewed from the exterior of the sphere.

	Finally, the Darboux indicatrix~$D$ of~$X$ should take into account that the tangent direction is constant along each edge of~$X$ and the binormal direction is constant at each vertex of~$X$.  To this end, we define the vertices~$D_i$ of~$D$ by
	\[D_{2i-1} = \begin{cases} T_i & \mbox{if } \tau_i > 0 \\ -T_i & \mbox{if } \tau_i < 0 \end{cases}, \quad D_{2i} = B_i \quad \mbox{for } i = 1, \ldots, n, \qquad \textrm{and} \qquad D_0 = D_{2n}.\]
Like the notrix, the Darboux indicatrix also has twice as many vertices as~$X$.  It is easy to see that under this definition, the vertices of the Darboux indicatrix are the axes of rotation of the Frenet frame.  The sign of the tangent vector at the vertices of the Darboux indicatrix is chosen so that its direction, together with the direction of rotation of the Frenet frame, form a right-handed screw, as we shall show in the course of proving Theorem~\ref{thm:inflection_polygonal}.  

\section{Duality Relationships for Indicatrices of Stick Knots} \label{sec:polygonal_knots}

	In this section, we discuss stick knot analogues of the maps defined by the various indicatrices and show that the graphs of these maps satisfy similar properties as in Sec.~\ref{sec:smooth_knots}.

	We begin with a discussion of how intersections of the various indicatrices of a stick knot with a great circle can be interpreted in terms of properties of the knot conformation.  Throughout, we shall use the convention that an intersection of an indicatrix with a great circle over an interval is counted as a single intersection.

	First, the tantrix arc $T_iT_{i+1}$ of a stick knot intersects the equatorial $xy$-plane, with neither of its vertices lying on the equator, if and only if the vertex~$X_i$ of the knot is a local extremum along the $z$-axis and neither of the edges $X_{i-1}X_i$ and $X_iX_{i+1}$ have stationary height along the $z$-axis, since a vertex of the tantrix above the equator corresponds to an edge of the knot directed upwards, while a vertex of the tantrix below the equator corresponds to an edge of the knot directed downwards.  Next, a vertex~$T_i$ of the tantrix lies on the equator if and only if the edge $X_{i-1}X_i$ of the knot has stationary height along the $z$-axis.  It follows that the number of intersections of the tantrix of a stick knot with a great circle counts the number of stationary points in the projection onto the axis perpendicular to the plane of the great circle, where a point that is stationary over an interval counts as a single stationary point.  This allows us to define a direct analogue of the bridge map of a smooth knot conformation, as follows.
	
\begin{definition}
	The \emph{bridge map} of a stick conformation~$K$ is a map defined on the unit sphere~$S^2$ by
	\[\mathfrak{b}_\mathbf{v}(K) = \#\{\mbox{stationary points in the projection of $K$ onto the axis defined by $\mathbf{v}$}\}\]
for each $\mathbf{v} \in S^2$, where we use the convention that a point that is stationary over an interval counts as a single stationary point.
	\label{def:stick_bridge}
\end{definition}

	Our argument above immediately yields the following lemma.

\begin{lemma}
	\label{lem:stick_tantrix}
	The bridge map is related to intersections of the tantrix with great circles by the following formula:
	\[\mathfrak{b}_\mathbf{v}(K) = \#\{\emph{intersections of the tantrix of $K$ with the great circle orthogonal to $\mathbf{v}$}\}.\]
\end{lemma}
	
	We require the following definition in studying intersections of the binotrix with great circles.

\begin{definition}
	Given a projection of a stick conformation, we define an \emph{inflection stick} to be (i) an edge of the stick conformation such that the two edges adjacent to it are projected to opposite sides of the line obtained by extending the projection of the edge infinitely (Fig.~\ref{fig:inflection_stick1}), (ii) a pair of adjacent edges projected to collinear sticks (Fig~\ref{fig:inflection_stick2}), or (iii) an edge projected down to a point (Fig~\ref{fig:inflection_stick3}).	
	\label{def:inflectionstick}
\end{definition}

\begin{figure}[ht]
	\centering
	\subfloat[Inflection stick]{\label{fig:inflection_stick1}\includegraphics[width=0.2\textwidth]{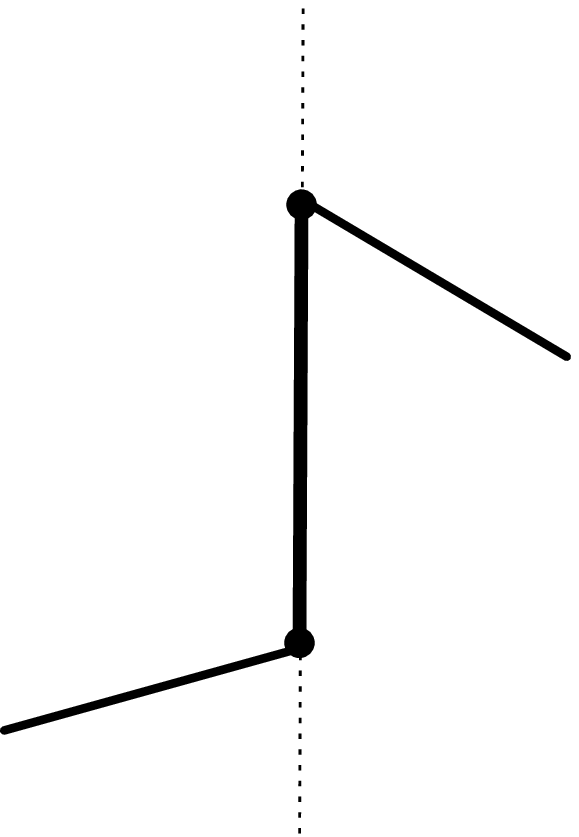}}\quad
	\subfloat[Inflection stick]{\label{fig:inflection_stick2}\includegraphics[width=0.2\textwidth]{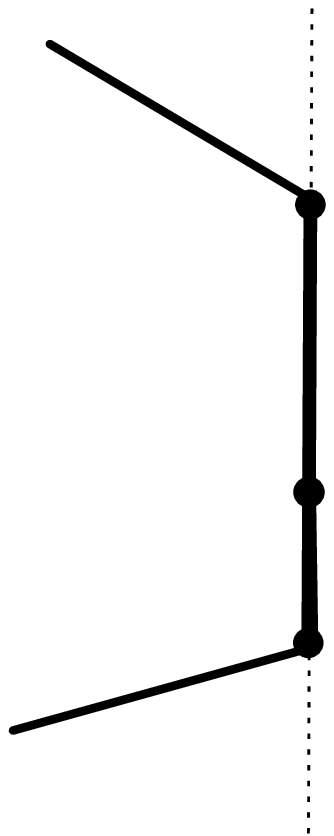}}\quad
	\subfloat[Inflection stick]{\label{fig:inflection_stick3}\includegraphics[width=0.2\textwidth]{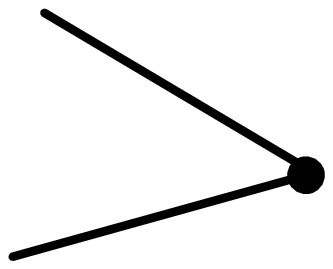}}\quad
	\subfloat[Not an inflection stick.]{\label{fig:not_inflection_stick}\includegraphics[width=0.2\textwidth]{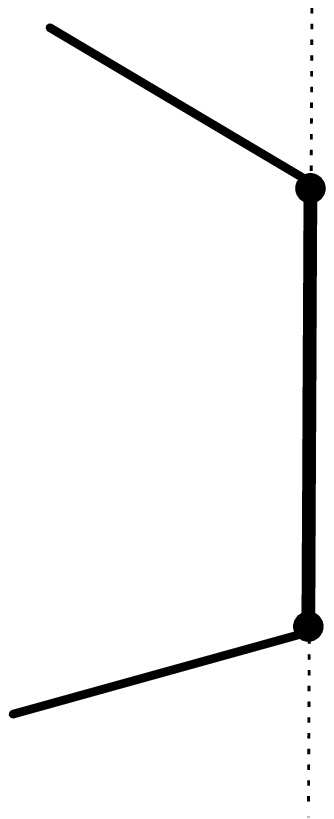}}\quad
	\caption{Examples and a non-example of an inflection stick.}
	\label{fig:inflection_stick_ex}
\end{figure}

	Inflection sticks of the first two types may be viewed as the stick knot equivalent of inflection points, while inflection sticks of the third type may be viewed as the stick knot equivalent of cusps.

\begin{definition}
	The \emph{inflection map} of a stick conformation~$K$ is the map defined on the unit sphere~$S^2$ by
	\[\mathfrak{i}_\mathbf{v}(K) = \#\{\mbox{inflection sticks in the projection of $K$ onto the plane orthogonal to $\mathbf{v}$}\}\]
for each $\mathbf{v} \in S^2$.
	\label{def:stick_inflection}
\end{definition}

\begin{lemma}
	\label{lem:stick_binotrix}
	The inflection map is related to intersections of the binotrix with great circles by the following formula:
	\[\mathfrak{i}_\mathbf{v}(K) = \#\{\emph{intersections of the binotrix of $K$ with the great circle orthogonal to $\mathbf{v}$}\}.\]
\end{lemma}

\begin{proof}
	Orient the sphere such that the plane orthogonal to $\mathbf{v}$ is the equatorial $xy$-plane and the projection of the edge $X_{i-1}X_i$ of the knot conformation~$K$ is in the direction of the positive $y$-axis.  We shall consider the effect of the position of the projections of the edges $X_{i-2}X_{i-1}$ and $X_iX_{i+1}$ on the binotrix of~$K$.

	Let the tangent vectors $T_{i-1}$, $T_i$ and $T_{i+1}$ be given by
	\[T_{i-1} = \left(\begin{array}{c} x_- \\ y_- \\ z_-\end{array}\right), \quad
	T_i = \left(\begin{array}{c} 0 \\ c \\ z_0 \end{array}\right), \quad
	T_{i+1} = \left(\begin{array}{c} x_+ \\ y_+ \\ z_+\end{array}\right).\]
where $c \geq 0$.
Then
	\[\widetilde{B}_{i-1} = T_{i-1} \times T_i = \left(\begin{array}{c} z_0y_- - cz_- \\ -z_0x_- \\ cx_-\end{array}\right), \quad
	\widetilde{B}_i = T_i \times T_{i+1} = \left(\begin{array}{c} cz_+ - z_0y_+ \\ z_0x_+ \\ -cx_+ \end{array}\right).\]
The binotrix arc $B_{i-1}B_i$ intersects the equatorial plane with neither of its vertices lying on the equator if and only if $c > 0$ and $x_-$ and $x_+$ have the same sign, so that we have the configuration in Fig.~\ref{fig:inflection_stick1}.  Next, the vertex $B_i$ of the binotrix lies on the equator but both of its adjacent vertices do not if and only if $c \neq 0$, $x_- \neq 0$, $x_+ = 0$, and $y_+ \neq 0$ (the final inequality arises from the fact that otherwise, $\widetilde{B}_{i+1} = T_{i+1} \times T_{i+2}$ would also lie on the equator).  Hence the edges $X_{i-1}X_i$ and $X_iX_{i+1}$ are projected down to parallel sticks, as in Fig.~\ref{fig:inflection_stick2}. Finally, $B_{i-1}$ and $B_i$ both lie on the equator if and only if $c = 0$ (we cannot have $x_- = x_+ = 0$ because of our assumption that the knot conformation is in general position), and thus the edge $X_{i-1}X_i$ is projected to a point, as in Fig.~\ref{fig:inflection_stick3}.  This proves the lemma.
\end{proof}

\begin{definition}
	The \emph{tantrix-bridge map} of a stick conformation~$K$ is a map defined on the unit sphere~$S^2$ by
	\[\mathfrak{tb}_\mathbf{v}(K) = \#\{\mbox{stationary points in projection of tantrix of $K$ onto axis defined by $\mathbf{v}$}\}\]
for each $\mathbf{v} \in S^2$, where we use the convention that a point that is stationary over an interval counts as a single stationary point.
	\label{def:stick_tantrix-bridge}
\end{definition}

\begin{lemma}
	The tantrix-bridge map is related to intersections of the notrix with great circles by the following formula:
	\[\mathfrak{tb}_\mathbf{v}(K) = \#\{\mbox{intersections of the notrix of $K$ with the great circle orthogonal to $\mathbf{v}$}\}.\]
\end{lemma}

\begin{proof}
	First, the tantrix arc $T_iT_{i+1}$ has an extrema in height along the axis defined by~$\mathbf{v}$ at a point~$p$ in its interior if and only if the notrix arc $N_{2i}N_{2i+1}$ intersects the great circle orthogonal to~$\mathbf{v}$ at the point~$p'$ on $N_{2i}N_{2i+1}$ obtained by moving~$p$ by a distance of~$\pi/2$ along the great circle that $T_iT_{i+1}$ lies on.  Next, the tantrix arc $T_{i-1}T_i$ is going up (respectively down) locally in the vicinity of $T_i$ if and only if the notrix vertex $N_{2i-1}$ lies above (respectively below) the equator, and the tantrix arc $T_iT_{i+1}$ is going down (respectively up) locally in the vicinity of~$T_i$ if and only if the notrix vertex~$N_{2i}$ lies below (respectively above) the equator.  It follows that the tantrix vertex~$T_i$ is an extrema in height if and only if $N_{2i-1}N_{2i}$ intersects the equator.  Since this accounts for all types of extrema in height and all types of intersections of the notrix with the equator, the lemma follows.	
\end{proof}

	Finally, we wish to interpret intersections of the Darboux indicatrix with a great circle in terms of projections of the tantrix to the plane defined by the great circle.  For convenience, we shall consider only projections where no vertex of the Darboux indicatrix lies on the great circle; we call such projections \emph{regular projections} and the corresponding intersections with the Darboux indicatrix \emph{regular intersections}.  The reader can easily generalize our formulae to include non-regular projections, if desired.  To this end, we consider the case where the great circle lies in the equatorial plane and study the effect of the position of vertices of the binotrix and tantrix on the projection of the the tantrix.  This is tabulated in Table~\ref{tab:stick_darboux}.  Note that the direction of rotation never changes along the projection of an arc of the tantrix because the arcs of the tantrix are great circle arcs.

\begin{table}[ht]
\begin{tabular}{|c|c|c|c|}
\hline & \textbf{Position w.r.t.} & & \\
\textbf{Vertex} & \textbf{equatorial plane} & \textbf{Torsion} $\tau_i$ & \textbf{Property of projection of tantrix} \\ \hline
$B_i$ & above & -- & tantrix arc $T_iT_{i+1}$ turns counterclockwise \\ \hline
$B_i$ & below & -- & tantrix arc $T_iT_{i+1}$ turns clockwise \\ \hline
$T_i$ & above & $> 0$ & angle to left between $T_{i-1}T_i$ and $T_iT_{i+1}$ $< \pi$ \\ \hline
$T_i$ & below & $> 0$ & angle to left between $T_{i-1}T_i$ and $T_iT_{i+1}$ $> \pi$ \\ \hline
$T_i$ & above & $< 0$ & angle to left between $T_{i-1}T_i$ and $T_iT_{i+1}$ $> \pi$ \\ \hline
$T_i$ & below & $< 0$ & angle to left between $T_{i-1}T_i$ and $T_iT_{i+1}$ $< \pi$ \\ \hline
\end{tabular}
\caption{Relation between the location of vertices of the binotrix and tantrix and the projection of the tantrix.}
\label{tab:stick_darboux}
\end{table}

	From Table~\ref{tab:stick_darboux}, we obtain the configurations shown in Table~\ref{tab:tantrix_projection} for projections of the tantrix corresponding to intersections of the Darboux indicatrix with the equator.

\begin{table}[ht]
\begin{tabular}{|c|c|}
\hline
\includegraphics[width=0.3\textwidth]{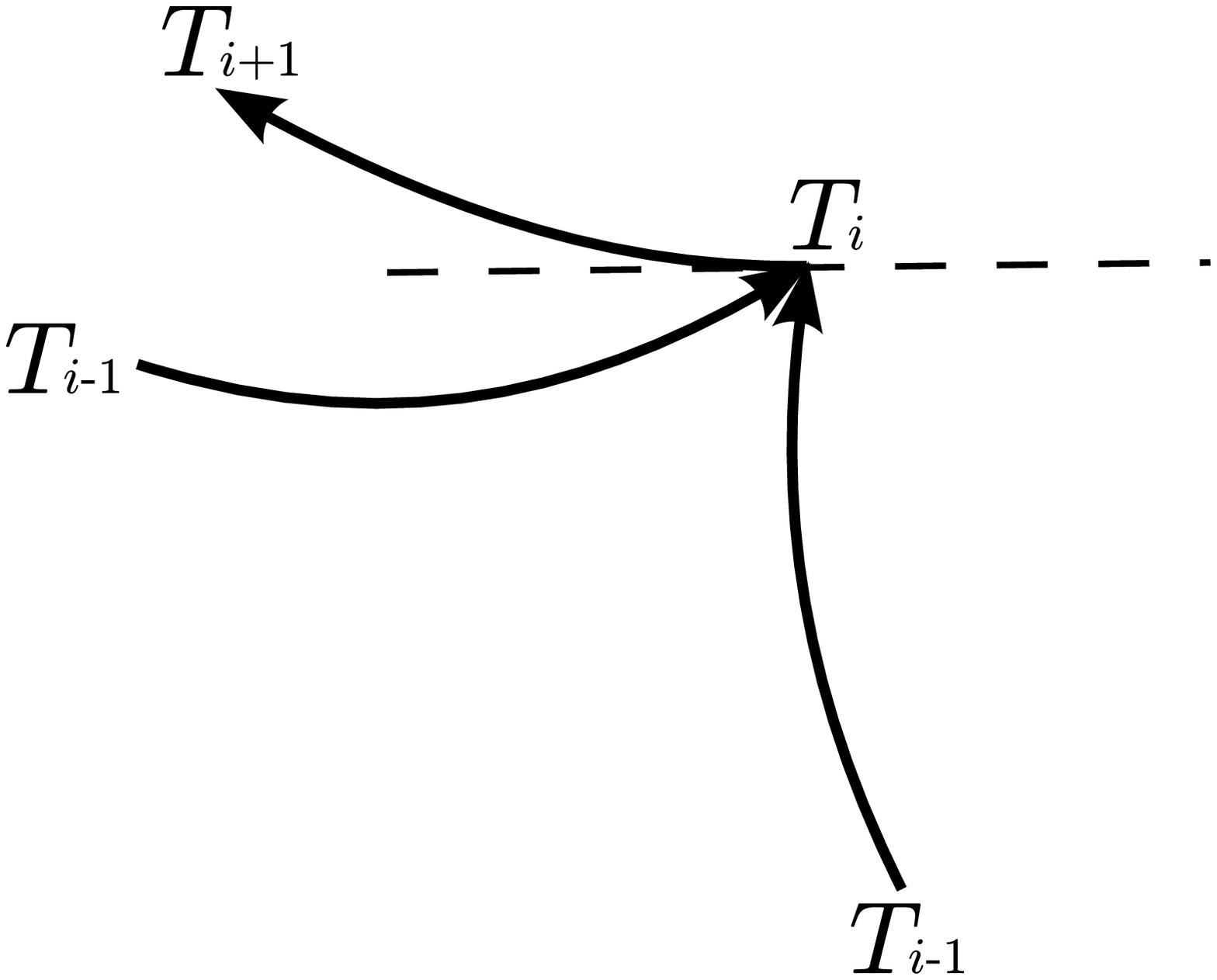} & \includegraphics[width=0.3\textwidth]{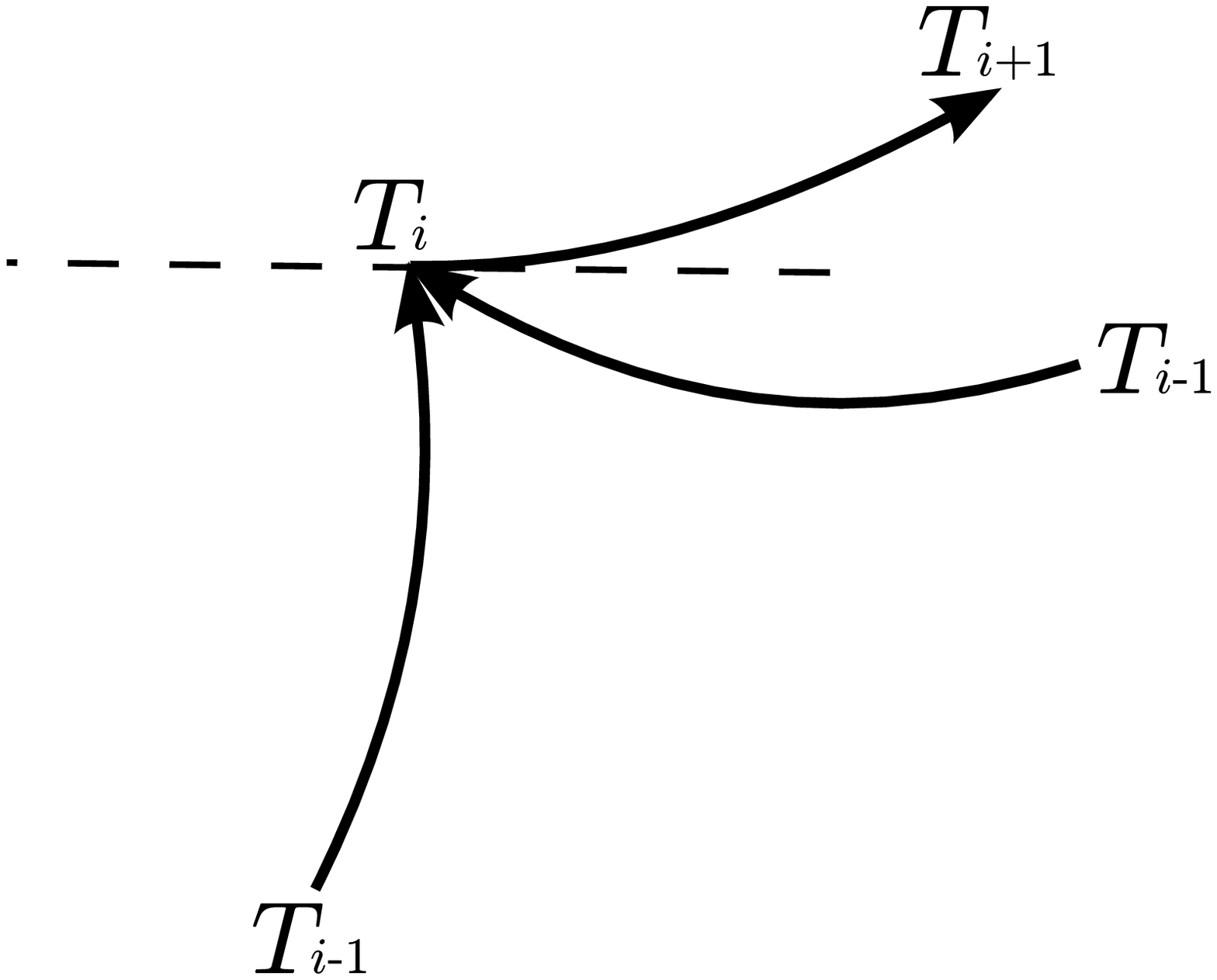}\\ 
$\tau_i > 0$, $D_{2i-1} = T_i$ above or & $\tau_i > 0$, $D_{2i-1} = T_i$ below or \\
$\tau_i < 0$, $D_{2i-1} = -T_i$ above; & $\tau_i < 0$, $D_{2i-1} = -T_i$ below; \\
$D_{2i} = B_i$ below & $D_{2i} = B_i$ above \\ \hline
\includegraphics[width=0.3\textwidth]{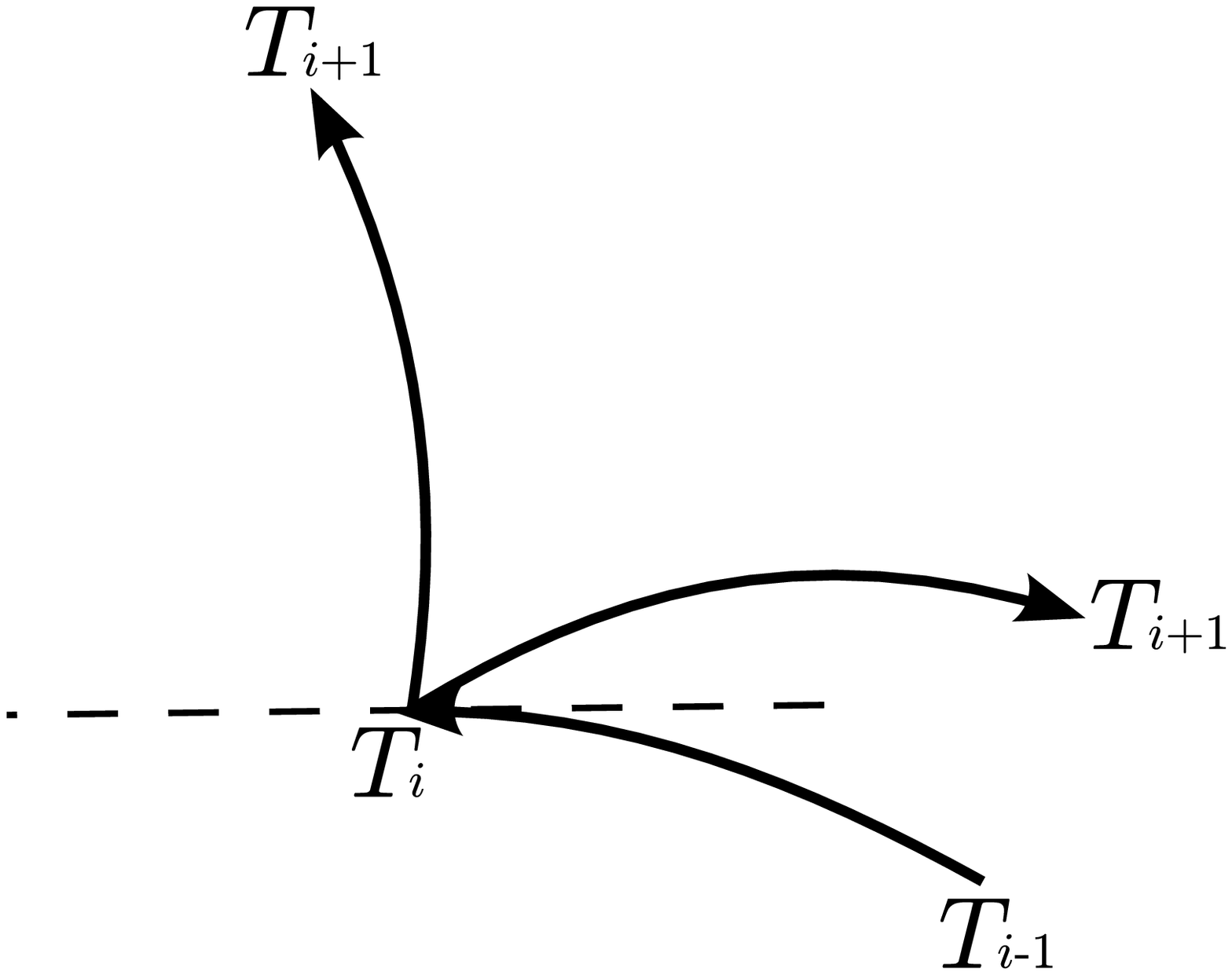} & \includegraphics[width=0.3\textwidth]{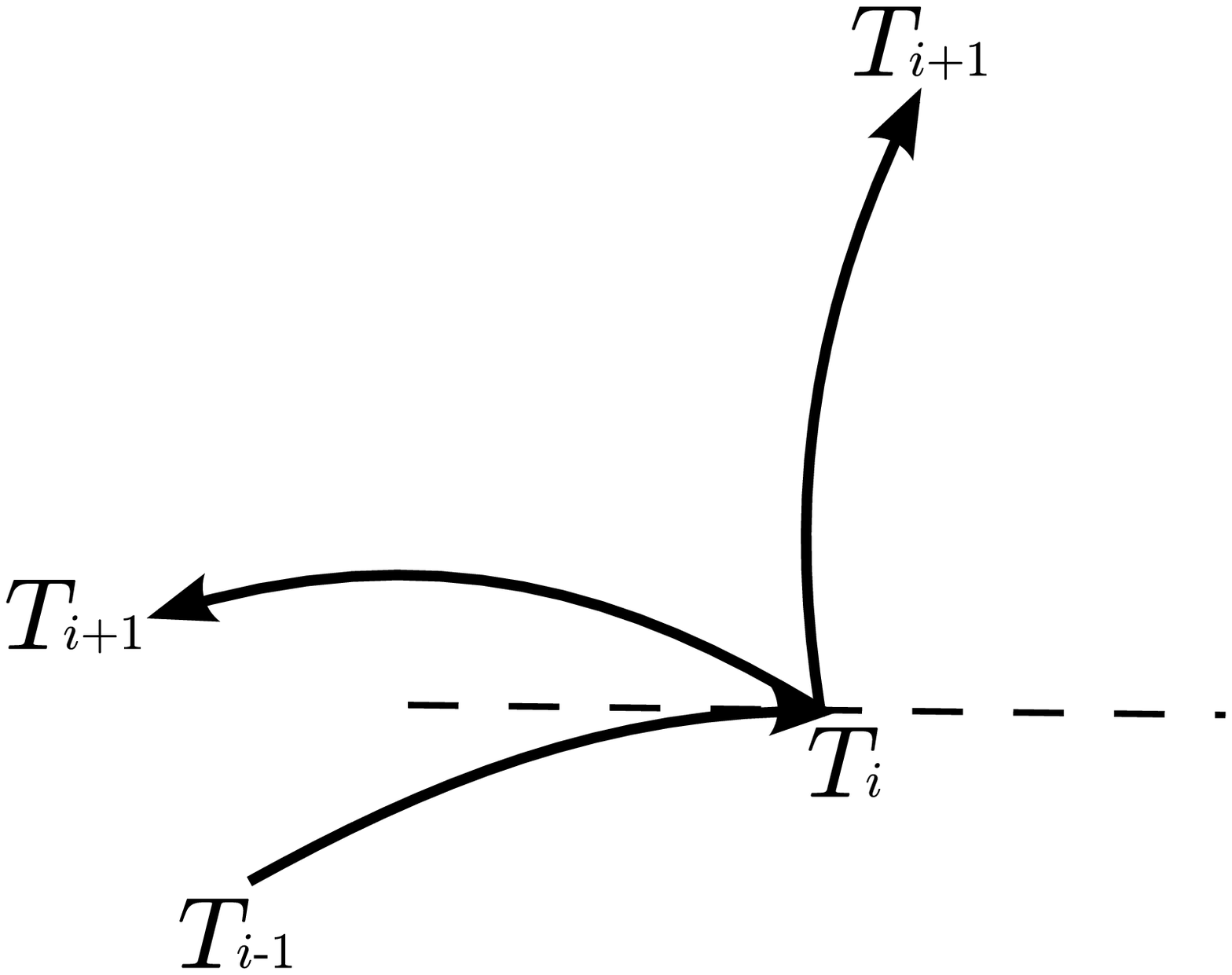}\\ 
$D_{2i} = B_i$ above; & $D_{2i} = B_i$ below; \\
$\tau_{i+1} > 0$, $D_{2i+1} = T_{i+1}$ below or & $\tau_{i+1} > 0$, $D_{2i+1} = T_{i+1}$ above or \\
$\tau_{i+1} < 0$, $D_{2i+1} = -T_{i+1}$ below & $\tau_{i+1} < 0$, $D_{2i+1} = -T_{i+1}$ above \\ \hline
\end{tabular}
\caption{Projection of tantrix when the Darboux indicatrix intersects the equator.}
\label{tab:tantrix_projection}
\end{table}

	This allows us to formulate the following definition.

\begin{definition}
	Let $d_a(\mathbf{v})$, $d_b(\mathbf{v})$ and $d_c(\mathbf{v})$ be the number of pairs of adjacent arcs of the tantrix whose projection onto the plane orthogonal to~$\mathbf{v}$ appear as in Figs.~\ref{fig:tantrix_projection_a}, \ref{fig:tantrix_projection_b}, and \ref{fig:tantrix_projection_c}  respectively.  The \emph{tantrix-inflection map} of a stick conformation $K$ is the map defined almost everywhere on the unit sphere $S^2$ by
	\[\mathfrak{ti}_\mathbf{v}(K) = d_a(\mathbf{v}) + d_b(\mathbf{v}) + 2d_c(\mathbf{v})\]
for each $\mathbf{v} \in S^2$ corresponding to a regular projection.
	\label{def:stick_tantrix_inflection}
\end{definition}

\begin{figure}[ht]
	\centering
	\subfloat[Pair of adjacent arcs contributing to $d_a(\mathbf{v})$]{\label{fig:tantrix_projection_a}\includegraphics[width=0.3\textwidth]{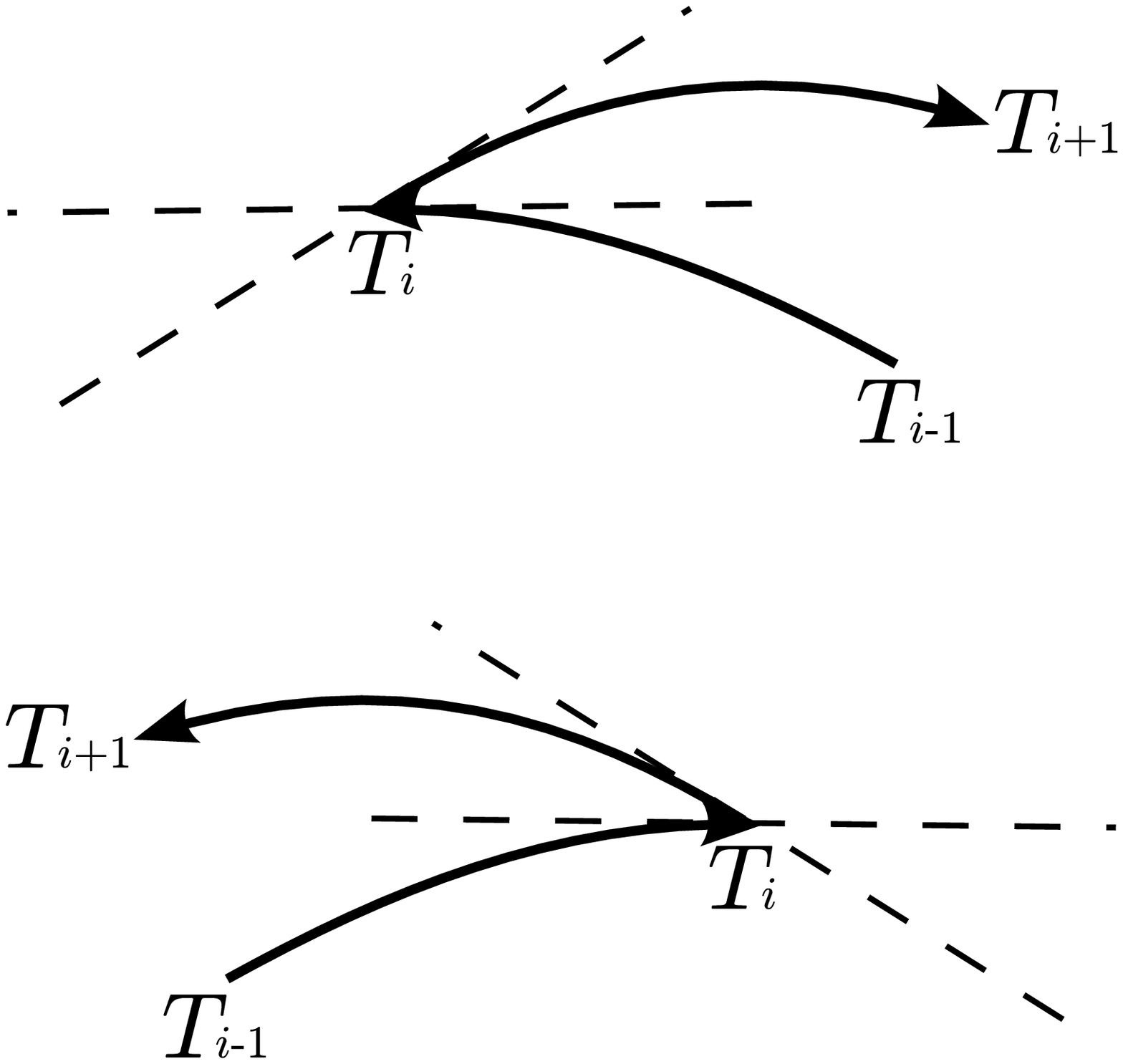}}\quad
	\subfloat[Pair of adjacent arcs contributing to $d_b(\mathbf{v})$]{\label{fig:tantrix_projection_b}\includegraphics[width=0.3\textwidth]{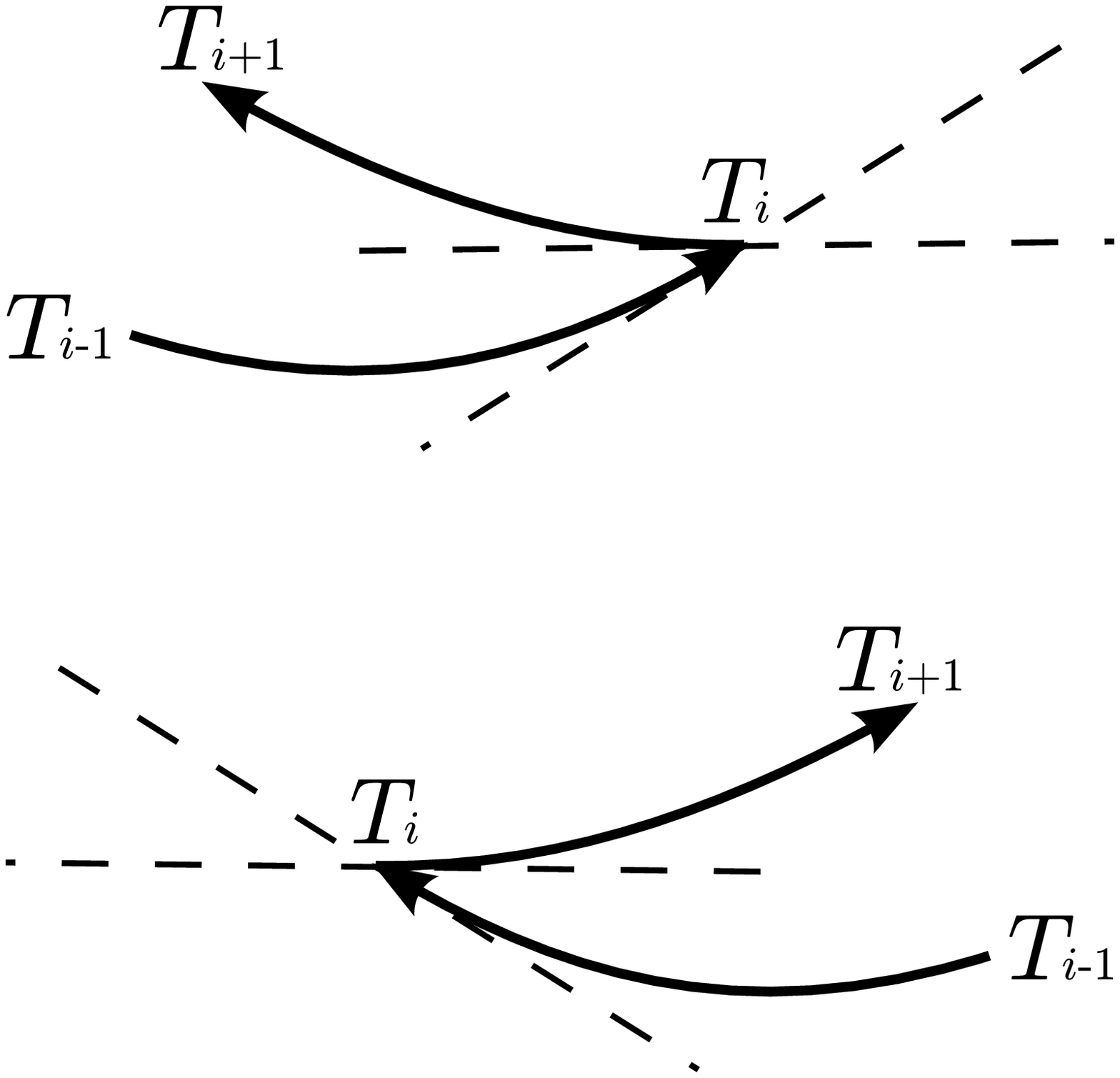}}\quad
	\subfloat[Pair of adjacent arcs contributing to $d_c(\mathbf{v})$]{\label{fig:tantrix_projection_c}\includegraphics[width=0.3\textwidth]{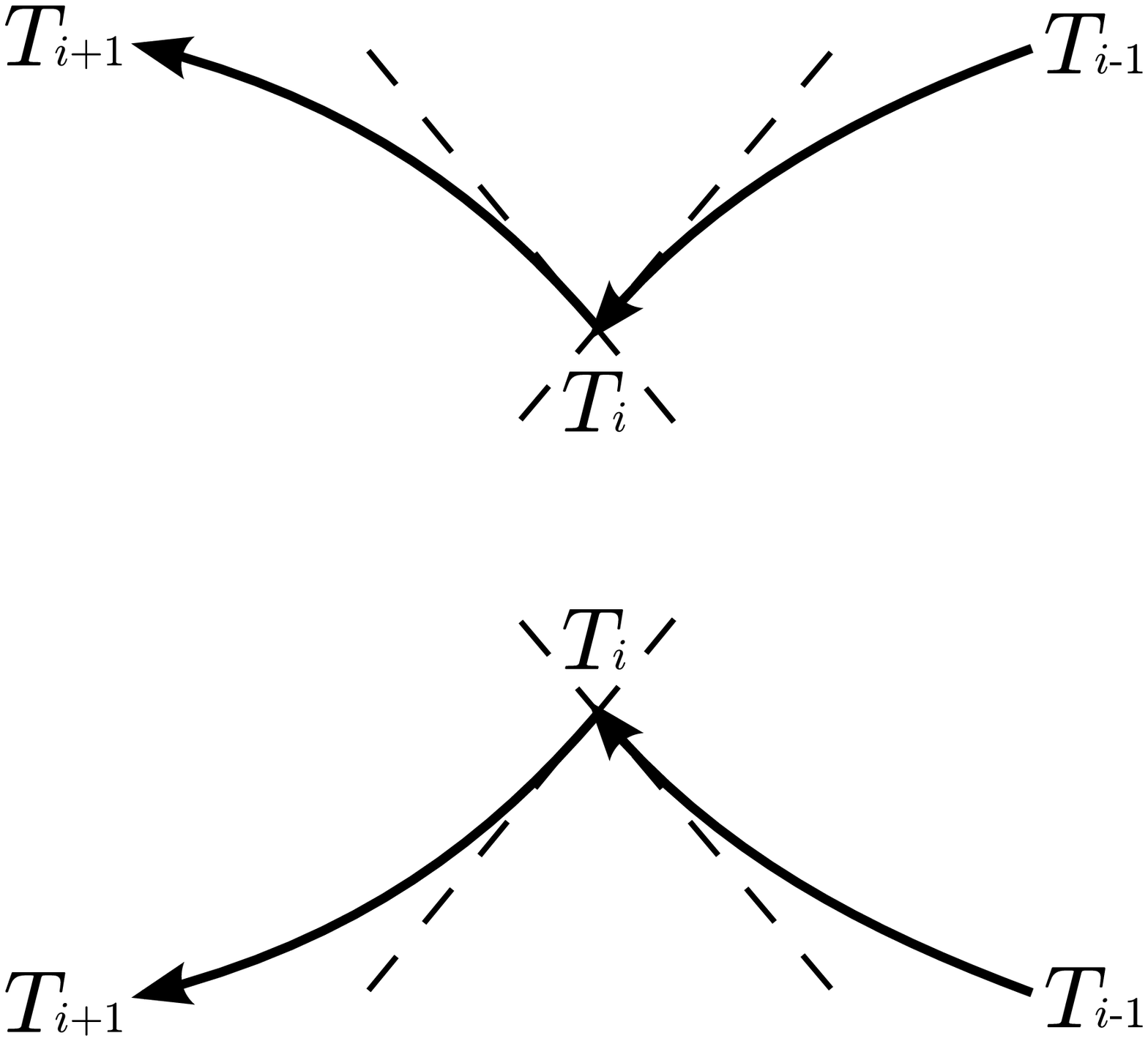}}\quad
	\caption{Types of projections of pairs of adjacent arcs that contribute to $d_a(\mathbf{v})$, $d_b(\mathbf{v})$, and $d_c(\mathbf{v})$ respectively.}
	\label{fig:tantrix_projection_types}
\end{figure}

	The reason for the coefficient~`2' in the above definition is explained in the following lemma.

\begin{lemma}
	\label{lem:stick_darboux}
	The tantrix-inflection map is related to intersections of the Darboux indicatrix with great circles by the following formula:
	\[\mathfrak{ti}_\mathbf{v}(K) = \#\{\emph{intersections of Darboux indicatrix of $K$ with the great circle orthogonal to $\mathbf{v}$}\}.\]
\end{lemma}

\begin{proof}
	As illustrated in Table~\ref{tab:tantrix_projection}, all regular intersections of the Darboux indicatrix with a great circle correspond to one of these diagrams. 	Each of the diagrams in Figs.~\ref{fig:tantrix_projection_a} and~\ref{fig:tantrix_projection_b} are counted by only one type of intersection of the Darboux indicatrix with a great circle (e.g., the upper diagram in Fig.~\ref{fig:tantrix_projection_a} is counted only by an arc of the Darboux indicatrix of the form $D_{2i}D_{2i+1}$ that goes from above the equator to below the equator, while the lower diagram in Fig.~\ref{fig:tantrix_projection_b} is counted only by an arc of the form $D_{2i-1}D_{2i}$ that goes from below the equator to above the equator).  On the other hand, both diagrams in Fig.~\ref{fig:tantrix_projection_c} are counted by two types of intersections of the Darboux indicatrix with a great circle (e.g. the upper diagram is counted by an arc of the form $D_{2i-1}D_{2i}$ that goes from below the equator to above the equator, as well as an arc of the form $D_{2i}D_{2i+1}$ that goes from above the equator to below the equator).  Hence the number of intersections of the Darboux indicatrix with the equator (and hence, with any great circle) is indeed given by the formula in Definition~\ref{def:stick_tantrix_inflection}.	
\end{proof}

	Having defined stick knot analogues of the maps defined by the various indicatrices, we now show that their graphs satisfy similar properties as in Sec.~\ref{sec:smooth_knots}.  As in Sec.~\ref{sec:smooth_knots}, we define the graph of a map on~$S^2$ to be the set of points $p \in S^2$ such that there does not exist an open neighborhood~$N_p$ of~$p$ on~$S^2$ (with the standard Euclidean topology) such that the value of the map is constant for all points $q \in N_p$ at which the map is defined.
	  
\begin{theorem}
	The bridge graph of a stick knot is the union of the binotrix and the anti-binotrix.
	\label{thm:bridge_polygonal}
\end{theorem}

\begin{theorem}
	The inflection graph of a stick knot is the spherical polygon obtained by connecting vertices of the tantrix and the anti-tantrix according to the following rule: we connect $T_i$ to $T_{i+1}$ and $-T_i$ to $-T_{i+1}$ if the torsions $\tau_i$ and $\tau_{i+1}$ associated to the edges $X_{i-1}X_i$ and $X_iX_{i+1}$ have the same sign, and we connect $T_i$ to $-T_{i+1}$ and $-T_i$ to $T_{i+1}$ if $\tau_i$ and $\tau_{i+1}$ have opposite signs.
	\label{thm:inflection_polygonal}
\end{theorem}

	Although the proofs of Theorems~\ref{thm:bridge_polygonal} and~\ref{thm:inflection_polygonal} are more involved than the proof of Theorem~\ref{thm:bridgeinflection}, they also provide more insight into the relationship between the tantrix, the binotrix and the sign of the torsion.

\begin{proof}[Proof of Theorem~\ref{thm:bridge_polygonal}]
	First, consider an intersection of the tantrix with a great circle in the vicinity of a vertex $T_i$ of the tantrix, as shown in Figs.~\ref{fig:tantrix_vertex} and ~\ref{fig:tantrix_vertex2}.  Recall that we assumed that no two (undirected) edges of the knot are parallel; this implies that no two vertices of the tantrix coincide and no vertex of the tantrix coincides with a vertex of the anti-tantrix.  By the definition of the binotrix, the vertices $B_{i-1}$ and~$B_i$ lie on the great circle orthogonal to~$T_i$ and always lie on the left side of the directed tantrix arcs $T_{i-1}T_i$ and $T_iT_{i+1}$ respectively as seen from from the exterior of the sphere.  The number of intersections of the tantrix with a great circle changes exactly when the center of the great circle moves across the open arc $B_{i-1}B_i$ and its antipodal arc, by our assumption on the vertices of the tantrix, thus contributing to the bridge graph all the open arcs of the binotrix.

\begin{figure}[ht]
	\centering
	\subfloat[]{\label{fig:tantrix_vertex}\includegraphics[width=0.3\textwidth]{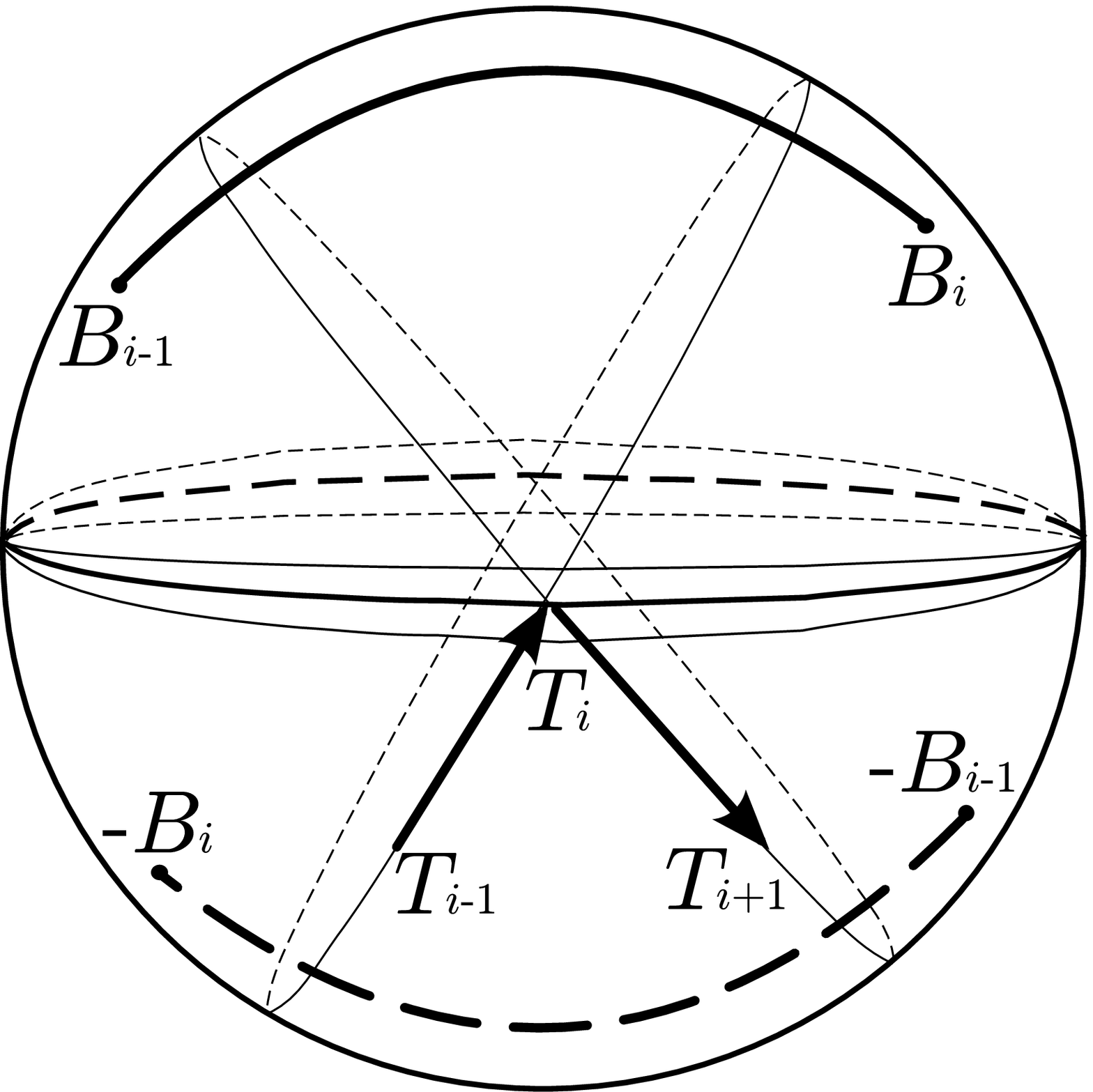}}\qquad                
	\subfloat[]{\label{fig:tantrix_vertex2}\includegraphics[width=0.3\textwidth]{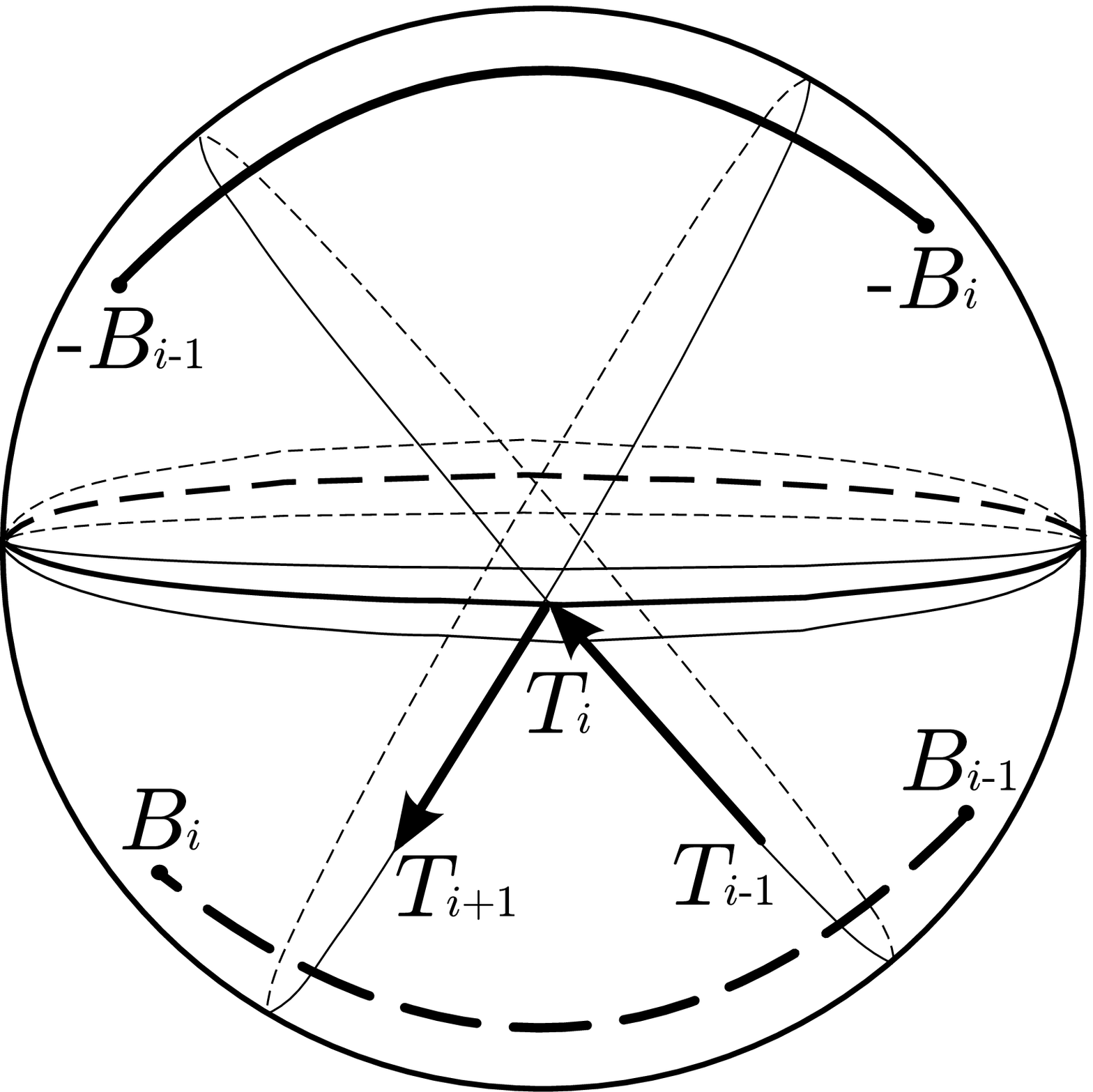}}
	\caption{Intersection of the tantrix with a great circle in the vicinity of a vertex of the tantrix.}
	\label{fig:tantrix_vertex_intersection}
\end{figure}

	Next, consider the situation when the great circle moves across an arc of the tantrix while parallel to it, as shown in Figs.~\ref{fig:tantrix_edge_up} and~\ref{fig:tantrix_edge_down}.  Orient the sphere such that the tantrix arc $T_iT_{i+1}$ lies on the equator, rotates counterclockwise as viewed from the north pole, and intersects the negative $y$-axis.

\begin{figure}[ht]
	\centering
	\subfloat[]{\label{fig:tantrix_edge_up}\includegraphics[width=0.3\textwidth]{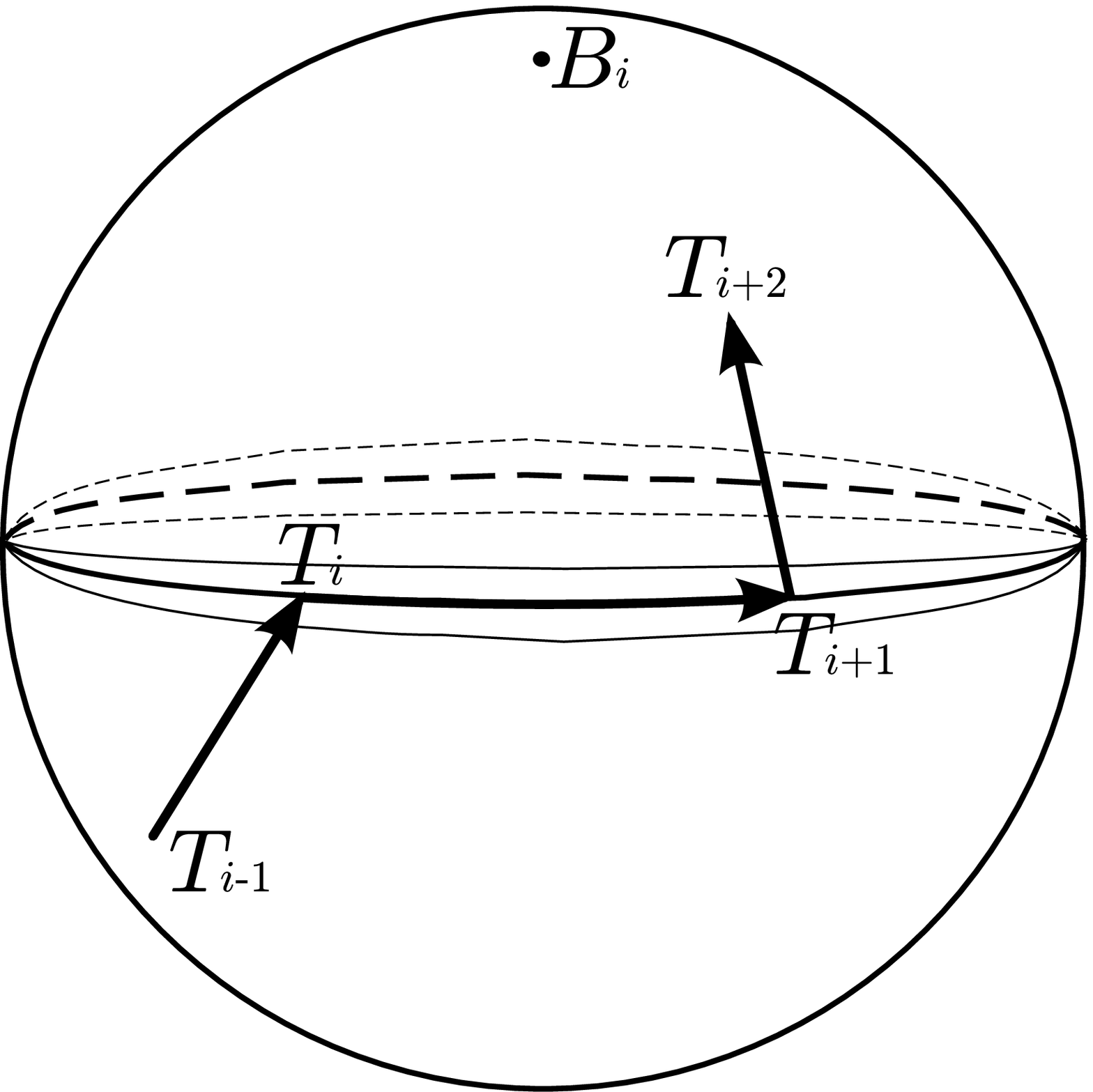}}\qquad                
	\subfloat[]{\label{fig:tantrix_edge_down}\includegraphics[width=0.3\textwidth]{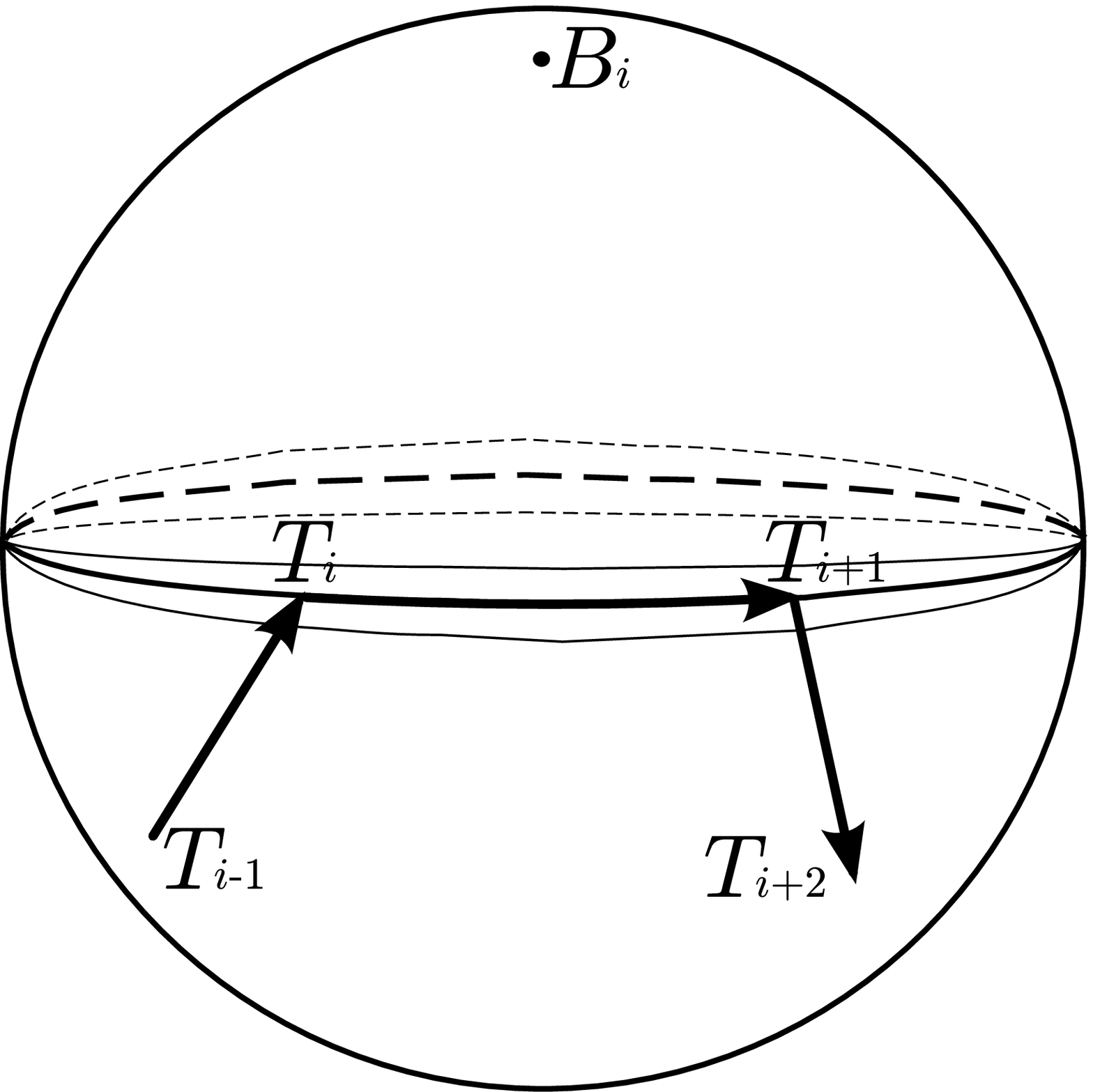}}
	\caption{Great circle moves across an arc of the tantrix while parallel to it.}
	\label{fig:tantrix_edge_intersection}
\end{figure}

	Since $B_{i-1}$ is perpendicular to~$T_i$, it lies on the great circle~$\Gamma_{i-1}$ dual to~$T_i$, and since $B_{i+1}$ is perpendicular to~$T_{i+1}$, it lies on the great circle~$\Gamma_{i+1}$ dual to~$T_{i+1}$, as illustrated in Fig.~\ref{fig:binotrix_top}, which shows a projection onto the equatorial plane.  Depending on whether the tantrix arcs $T_{i-1}T_i$ and $T_{i+1}T_{i+2}$ lie above or below the equator, the vertices $B_{i-1}$ and $B_{i+1}$ of the binotrix lie on different sides of~$B_i$, as shown in Fig.~\ref{fig:binotrix_top} (this can easily be seen by computing cross products and considering their $y$-coordinates).  In particular, we see that the binotrix arcs $B_{i-1}B_i$ and $B_iB_{i+1}$ lie on the same side of the $yz$-plane if and only if $T_{i-1}T_i$ and $T_{i+1}T_{i+2}$ lie on the same side of the equator (see Figs.~\ref{fig:binotrix_top_ex} and~\ref{fig:binotrix_top_ex2} for examples).  Hence the number of intersections of the tantrix with a great circle changes only when we move into a different region defined by the binotrix.  This completes the proof.

\begin{figure}[ht]
	\centering \includegraphics[width=0.3\textwidth]{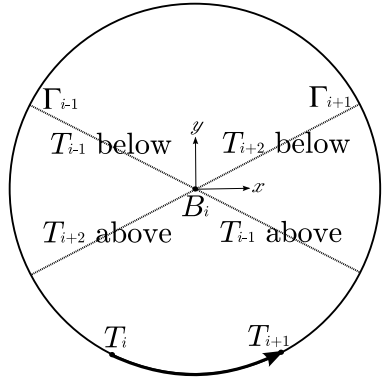}
	\caption{Location of $B_{i-1}$ and $B_{i+1}$ depending on whether the tantrix arcs $T_{i-1}T_i$ and $T_{i+1}T_{i+2}$ lie above or below the equator, as seen from a projection onto the equatorial plane.  Note that the great circles $\Gamma_{i-1}$ and $\Gamma_{i+1}$ are projected onto line segments and that $B_{i-1}$ and $B_{i+1}$ can lie on either side of the equator.}
	\label{fig:binotrix_top}
\end{figure}

\begin{figure}[ht]
	\centering
	\subfloat[]{\label{fig:binotrix_top_ex}\includegraphics[width=0.3\textwidth]{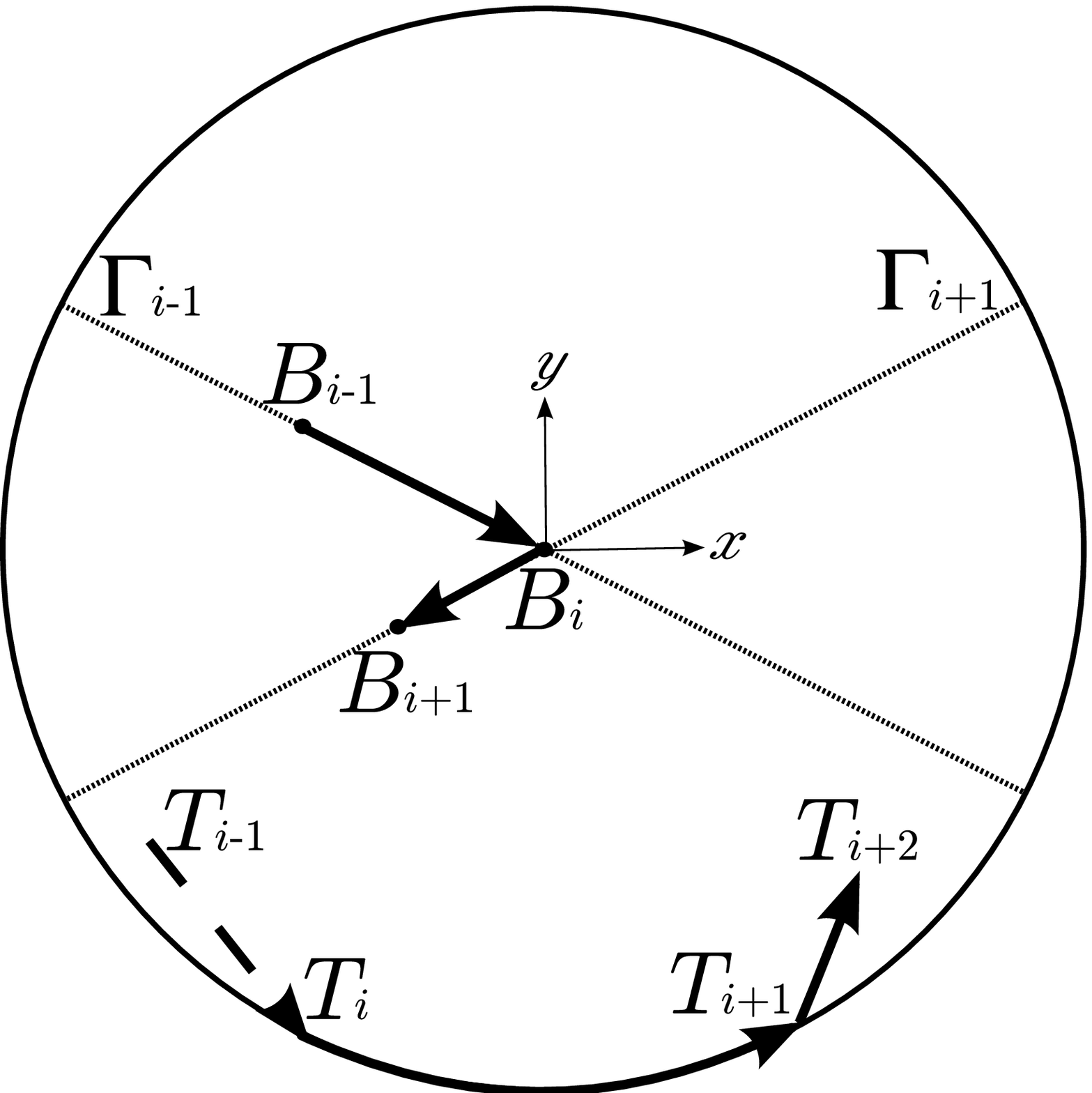}}\qquad                
	\subfloat[]{\label{fig:binotrix_top_ex2}\includegraphics[width=0.3\textwidth]{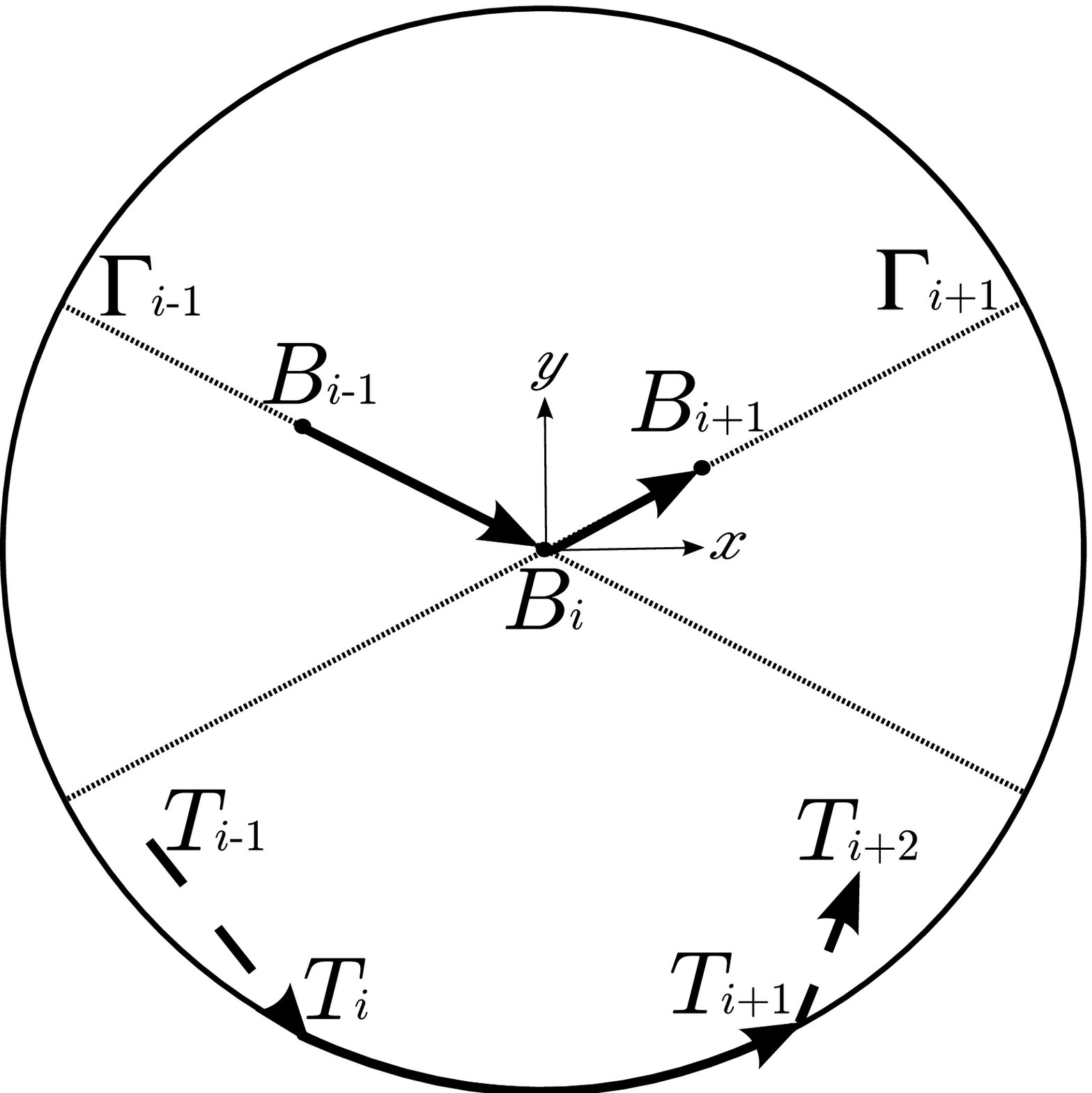}}
	\caption{Examples showing how the binotrix arcs $B_{i-1}B_i$ and $B_iB_{i+1}$ lie on the same side of the $yz$-plane if and only if $T_{i-1}T_i$ and $T_{i+1}T_{i+2}$ lie on the same side of the equator.}
	\label{fig:binotrix_top_examples}
\end{figure}

\end{proof}

\begin{remark}
	Theorem~\ref{thm:bridge_polygonal} can be re-interpreted in the following manner.  Let~$P$ be a directed spherical polygon with sides $(l_0, l_1, \ldots, l_{n-1})$, such that no two of its vertices coincide and none of its vertices coincides with a vertex of its antipodal spherical polygon.  For each geodesic segment~$l_i$, let~$w_i$ be the point on~$S^2$ obtained by moving a distance of~$\pi/2$ normal to the geodesic segment and towards its left as viewed from the exterior of the sphere.  Let~$P^*$ be the spherical polygon with vertices $(w_0, w_1, \ldots, w_{n-1})$.  Then the graph of the map on~$S^2$ that assigns to each vector $\mathbf{v} \in S^2$ the number of intersections of~$P$ with the great circle orthogonal to~$\mathbf{v}$ is given by the union of~$P^*$ and its antipodal spherical polygon.
	
	In particular, if~$P$ is a directed co-oriented spherical polygon such that the co-orienting normal points towards the left when viewed from the exterior of the sphere, then the vertices of the dual spherical polygon are precisely the~$w_i$'s and hence the graph of the map on~$S^2$ will simply be the union of~$P^\vee$ and its antipodal spherical polygon.
	\label{rem:reinterpretation}
\end{remark}

\begin{proof}[Proof of Theorem~\ref{thm:inflection_polygonal}]
	Recall that we assumed that no two osculating planes of~$X$ are parallel, so that no two vertices of the binotrix coincide and no vertex of the binotrix coincides with a vertex of the anti-binotrix.  By Remark~\ref{rem:reinterpretation}, it suffices to show that the vertex obtained by moving a distance of~$\pi/2$ normal to the binotrix arcs and towards their left as viewed from the exterior of the sphere is~$T_i$ when the torsion~$\tau_i$ is positive, and~$-T_i$ when the torsion~$\tau_i$ is negative. Orient the sphere such that the tantrix arc $T_{i-1}T_i$ lies on the equatorial $xy$-plane and rotates counterclockwise as viewed from the north pole, and $T_i$ lies on the negative $y$-axis.  Since $B_i$ is orthogonal to $T_i$, it lies on the unit circle in the $xz$-plane.  If $\tau_i$~is positive, then $T_{i+1}$ lies above the equator and thus the $x$-coordinate of $B_i$ is negative, as illustrated in Fig.~\ref{fig:sideofbinotrix_up}.  Hence $T_i$ lies on the left side of $B_{i-1}B_i$.  On the other hand, if $\tau_i$~is negative, then $T_{i-1}$ lies below the equator and thus the $x$-coordinate of $B_i$ is positive, as shown in Fig.~\ref{fig:sideofbinotrix_down}.  Hence $T_i$ lies on the right side of $B_{i-1}B_i$, that is, $-T_i$ lies on the left side of $B_{i-1}B_i$, as required.

\begin{figure}[ht]
	\centering
	\subfloat[$\tau_i > 0$]{\label{fig:sideofbinotrix_up}\includegraphics[width=0.3\textwidth]{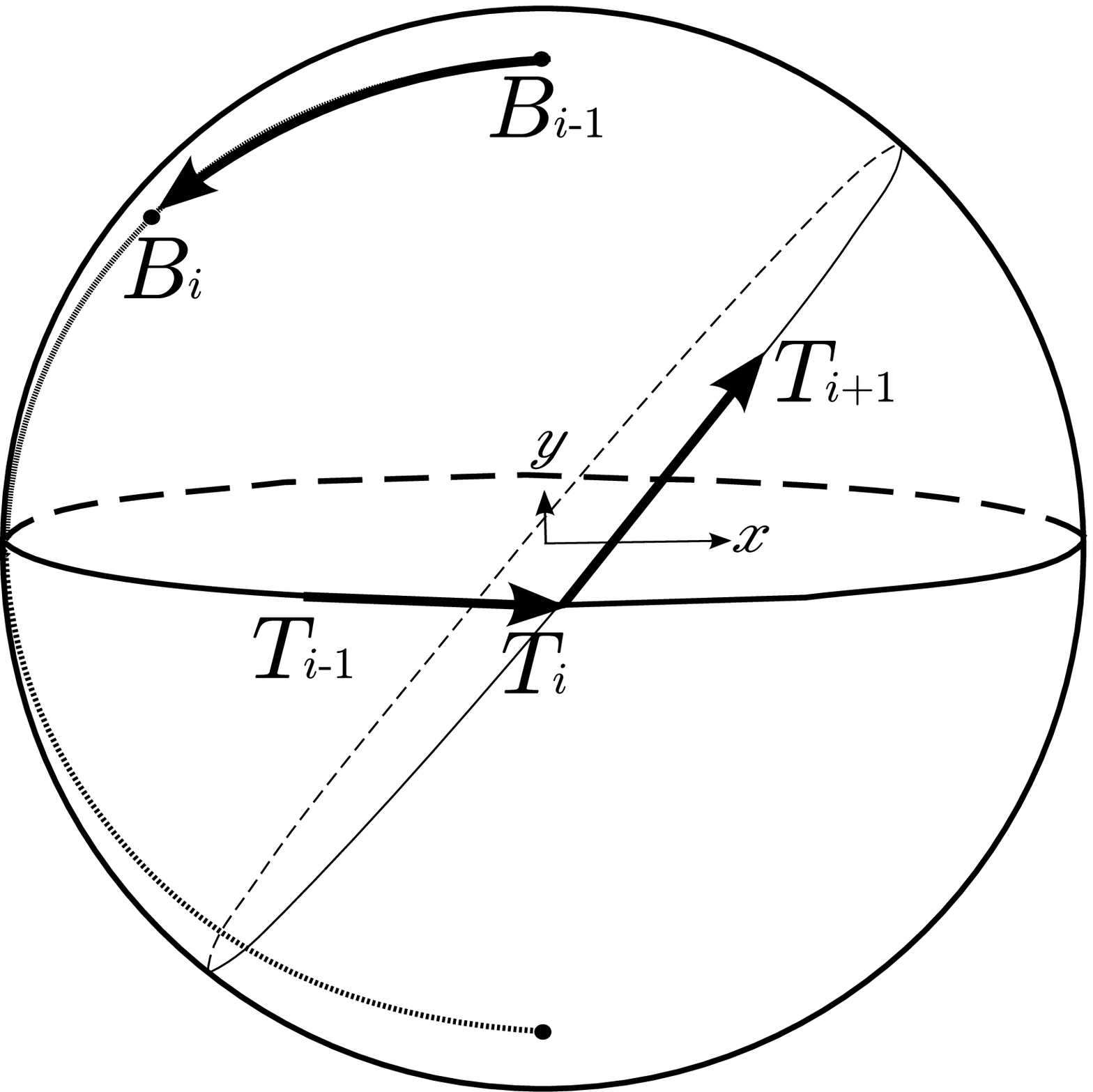}}\qquad        
	\subfloat[$\tau_i < 0$]{\label{fig:sideofbinotrix_down}\includegraphics[width=0.3\textwidth]{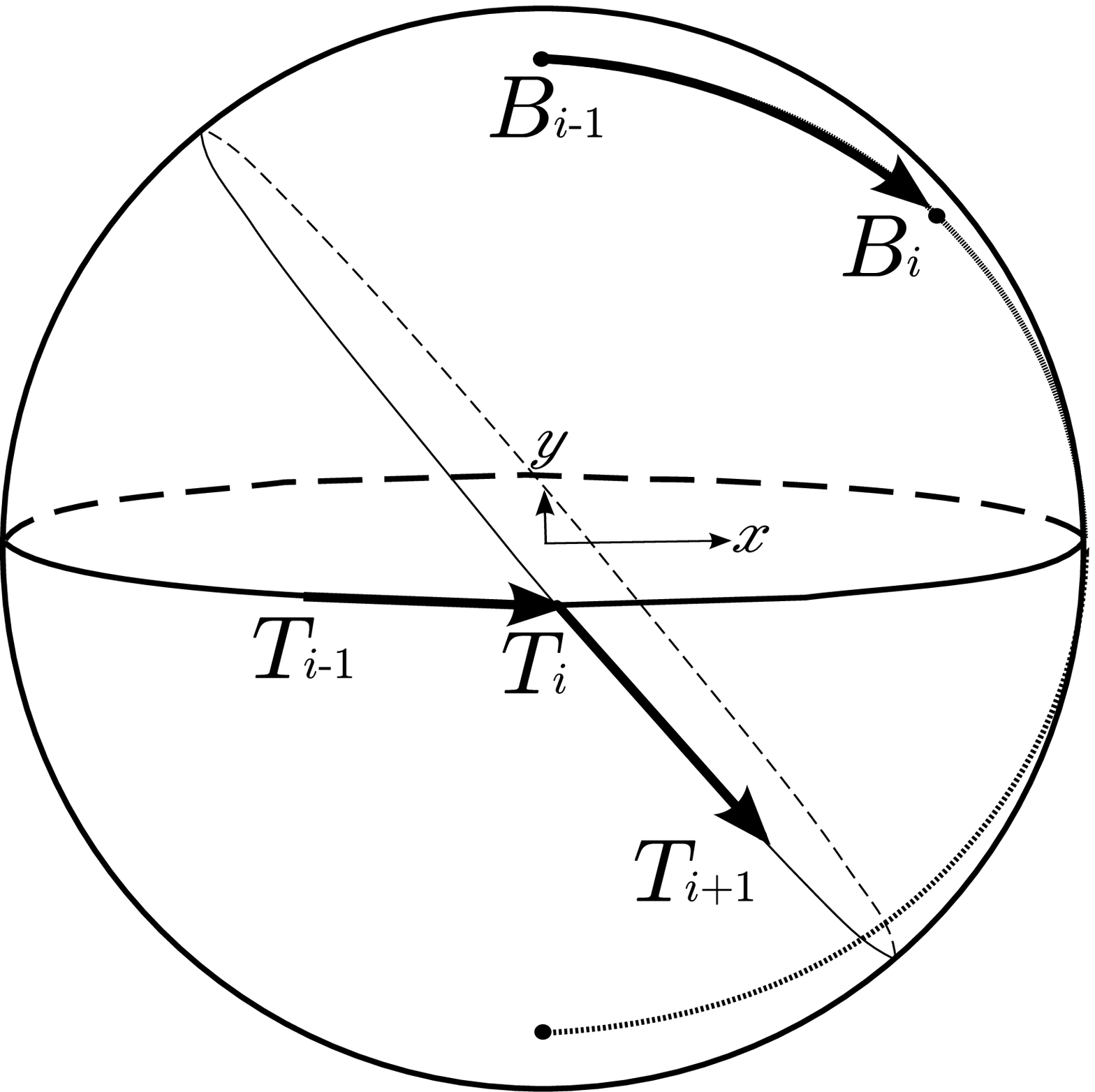}}
	\caption{Dependence of the $x$-coordinate of $B_i$ on the torsion $\tau_i$ and its effect on the side of $B_{i-1}B_i$ that $T_i$ lies on.}
	\label{fig:sideofbinotrix}
\end{figure}

	Alternatively, this can also be seen from the vector triple product formula.  We have
	\[\widetilde{B}_{i-1} \times \widetilde{B}_i = (T_{i-1} \times T_i) \times (T_i \times T_{i+1}) = ((T_{i-1} \times T_i) \cdot T_{i+1})T_i.\]
Since the sign of $(T_{i-1} \times T_i) \cdot T_{i+1}$ is precisely the sign of~$\tau_i$, we see that the vertex obtained by moving a distance of $\pi/2$ normal to $B_{i-1}B_i$ and towards the left of $B_{i-1}B_i$ is~$T_i$ when $\tau_i$~is positive, and $-T_i$ when $\tau_i$~is negative, as required.
\end{proof}

\begin{remark}
	Theorem~\ref{thm:inflection_polygonal} provides an alternative interpretation of the great circles tangent to the tantrix and anti-tantrix at $\tau$-inflection points of a smooth knot in Theorem~\ref{thm:bridgeinflection}.  As we remarked earlier, we can talk about a change in the sign of the torsion at a vertex~$B_i$ of the binotrix of a stick knot if $\tau_i$ and~$\tau_{i+1}$ have different signs.  The arcs of the inflection graph corresponding to such a vertex~$B_i$ join $T_i$ to $-T_{i-1}$ and $-T_i$ to $T_{i+1}$.  Similarly, $\tau$-inflection points of a smooth knot are points where the sign of the torsion changes, and the great circles tangent to the tantrix and anti-tantrix can be viewed as the union of two antipodal great circle arcs joining the tantrix to the antitantrix.

	In the case of the bridge graph of a stick knot, all the arcs connect two adjacent vertices of the binotrix and two adjacent vertices of the anti-binotrix.  This reflects the fact that the sign of curvature, as defined in Sec.~\ref{sec:preliminaries_smooth}, unlike that of torsion, has no geometric meaning, and serves merely to allow us to have isolated points of zero curvature on a smooth knot.
	\label{rem:greatcirclestangent}
\end{remark}

	The proof of Theorem~\ref{thm:inflection_polygonal} easily yields the following proposition.

\begin{proposition}
	The Darboux indicatrix of a stick knot is the dual of its notrix.
	\label{prop:notrix_darboux_dual}
\end{proposition}

\begin{proof}
	We only need to verify that the points dual to the arcs of the notrix are the vertices of the Darboux indicatrix.  This follows immediately from the equations
	\[N_{2i-1} \times N_{2i} = (T_i \times B_{i-1}) \times (T_i \times B_i) = B_{i-1} \times B_i = \frac{(T_{i-1} \times T_i) \cdot T_{i+1}}{\norm{T_{i-1} \times T_i} \cdot \norm{T_i \times T_{i+1}}} T_i\]
and
	\[N_{2i} \times N_{2i+1} = (T_i \times B_i) \times (T_{i+1} \times B_i) = T_i \times T_{i+1} = \widetilde{B}_i,\]
and the fact that $(T_{i-1} \times T_i) \cdot T_{i+1}$ has the same sign as the torsion~$\tau_i$.
\end{proof}

	Remark~\ref{rem:reinterpretation} immediately yields the following corollary to Proposition~\ref{prop:notrix_darboux_dual}.

\begin{corollary}
	The tantrix-bridge graph of a stick knot is the union of the Darboux indicatrix and the anti-Darboux indicatrix.
	\label{cor:tantrixbridge_polygonal}
\end{corollary}

	Finally, the following theorem is the stick equivalent of the second half of Theorem~\ref{thm:tantrix-bridgeinflection}.

\begin{theorem} The tantrix-inflection graph of a stick knot is the spherical polygon obtained by connecting vertices of the notrix and the anti-notrix according to the following rule: we always connect $N_{2i-1}$ to $-N_{2i}$ and $-N_{2i-1}$ to $N_{2i}$, we connect $N_{2i}$ to $-N_{2i+1}$ and $-N_{2i}$ to $N_{2i+1}$ if $\tau_i$ and $\tau_{i+1}$ have the same sign, and we connect $N_{2i}$ to $N_{2i+1}$ and $-N_{2i}$ to $-N_{2i+1}$ if $\tau_i$ and $\tau_{i+1}$ have different signs. \label{thm:tantrixinflection_polygonal} \end{theorem}

	That is, if two adjacent edges of the Darboux indicatrix share a vertex of the tantrix, then we always join the notrix vertex corresponding to one edge to the anti-notrix vertex corresponding to the other edge, and if two adjacent edges of the Darboux indicatrix share a vertex of the binotrix, then we join the notrix vertex corresponding to one edge to either the notrix vertex or the anti-notrix vertex corresponding to the other edge depending on whether or not the sign of the torsion changes at that vertex of the binotrix.

\begin{proof}
	Since the Darboux indicatrix is the dual of the notrix, each of the points obtained by moving a distance of~$\pi/2$ normal to each arc of the Darboux indicatrix and towards its left as viewed from the exterior of the sphere is either a vertex of the notrix or the anti-notrix.  By Remark~\ref{rem:reinterpretation}, it suffices to check that this sequence of vertices satisfies the conditions above.  We verify this by explicitly computing this sequence of vertices:
	\begin{alignat*}{8}
	D_{2i-1} & \times D_{2i} & \; & & = & \; & T_i & \times B_i & \; & & = & \; & -&N_{2i} && \mbox{if $\tau_i > 0$}, \\
	D_{2i-1} & \times D_{2i} && & = && -T_i & \times B_i && & = && &N_{2i} && \mbox{if $\tau_i < 0$}, \\
	D_{2i} & \times D_{2i+1} && & = && B_i & \times T_{i+1} && & = && &N_{2i+1} && \mbox{if $\tau_{i+1} > 0$}, \\
	D_{2i} & \times D_{2i+1} && & = && B_i & \times -T_{i+1} && & = && -&N_{2i+1} & \;\; & \mbox{if $\tau_{i+1} > 0$}.
	\end{alignat*}
It is straightforward to see that when two adjacent edges $D_{2i-2}D_{2i-1}$, $D_{2i-1}D_{2i}$ of the Darboux indicatrix share a vertex~$T_i$ of the tantrix, we join $N_{2i-1}$ to $-N_{2i}$ and $-N_{2i-1}$ to $N_{2i}$, and when two adjacent edges $D_{2i-1}D_{2i}$, $D_{2i}D_{2i+1}$ share a vertex~$B_i$ of the binotrix, we join $N_{2i}$ to $N_{2i+1}$ if the sign of the torsion changes at that vertex of the binotrix, and to $-N_{2i+1}$ if the sign of the torsion does not change at that vertex of the binotrix.
\end{proof}

\begin{remark}
	As in Remark~\ref{rem:greatcirclestangent}, Theorem~\ref{thm:tantrixinflection_polygonal} provides an alternative interpretation of the great circles tangent to the notrix and anti-notrix of a smooth knot at points where the geodesic curvature~$\tau/\kappa$ of the tantrix is stationary in Theorem~\ref{thm:tantrix-bridgeinflection}.  If we interpret arcs of a great circle as having zero geodesic curvature, a vertex of the tantrix with negative torsion as having positive geodesic curvature, and a vertex of the tantrix of positive torsion as having negative geodesic curvature, then we have the equivalent of a ``point of stationary geodesic curvature of the tantrix'' at every vertex of the tantrix and along edges of the tantrix (which correspond naturally with vertices of the binotrix) where the sign of the torsion is the same at both vertices.  These are precisely the vertices of the Darboux indicatrix where we move from a vertex of the notrix to a vertex of the anti-notrix.  In this sense, the great circles tangent to the notrix and anti-notrix at points where the geodesic curvature $\tau/\kappa$ of the tantrix is stationary can be viewed as the union of two antipodal great circle arcs joining the notrix to the anti-notrix.
\end{remark}

\section{Discussion} \label{sec:discussion}

	Let~$K$ be a stick knot with $n$ edges.  Since each arc of a spherical indicatrix of~$K$ is a geodesic, it can intersect any great circle at most once.  It follows that the bridge and inflection maps of~$K$ satisfy $\mathfrak{b}_\mathbf{v}(K) \leq n$ and $\mathfrak{i}_\mathbf{v}(K) \leq n$ and the tantrix-bridge and tantrix-inflection maps of~$K$ satisfy $\mathfrak{tb}_\mathbf{v}(K) \leq 2n$ and $\mathfrak{ti}_\mathbf{v}(K) \leq 2n$ for all $\mathbf{v} \in S^2$.  We believe that for any stick knot~$K$, it is possible to use the methods in a paper by Wu~\cite{Wu02} to show that there exists a smooth knot $K'$ with the same knot type such that the maps defined by the various indicatrices of~$K'$ are combinatorially the same as those of $K$, that is, there exists an isotopy of~$S^2$ that sends a map on~$K'$ to that on~$K$ (note that we do not require the same isotopy for all maps).  This would enable us to obtain bounds on the stick number of knots by studying the maps defined by the various indicatrices for smooth knots.  Also, since each arc of the Darboux indicatrix of a stick knot connects a vertex of the tantrix to a vertex of the binotrix and thus has length~$\pi/2$, the total length of the Darboux indicatrix of~$K$ is~$n\pi$.  It would be interesting to see if one could relate the length of the Darboux indicatrix of a stick knot to the length of the Darboux indicatrix of smooth knot conformations of the same knot type, as this would also enable us to obtain bounds on the stick number of knots by studying their Darboux indicatrices.
	
	Observe that the bridge index and superbridge index of a knot can always be realized in smooth conformations with non-vanishing curvature.  In such a case, the bridge graph is simply the union of the binotrix and the anti-binotrix.  It follows that $sb[K] = b[K] + 1$ if and only if there exists a knot conformation $K \in [K]$ such that the binotrix does not intersect itself or the anti-binotrix and this conformation realizes the lowest possible value of the bridge map over all knots $K \in [K]$.  This is because at a point where the binotrix intersects itself or the anti-binotrix, there will be two regions adjacent to that point where the value of the bridge map differs by $4$, and hence the superbridge index would be at least two greater than the bridge index.  
	
	In particular, Jeon and Jin~\cite{Je02} conjectured that the only knots with superbridge index 3 are the trefoil knot and the figure eight knot.  This conjecture would be proved if it could be shown that the trefoil knot and the figure eight knot are the only non-trivial knots with conformations such that the binotrix does not intersect itself or the anti-binotrix and  the bridge map realizes a value of $4$ in such a conformation.

\end{document}